\documentclass{amsart}


\title[]
          {Shadowing, Topological entropy and Recurrence of induced Morse-Smale diffeomorphism}

\author[A. Arbieto]{A. Arbieto}
\address{Instituto de Matem\'atica, Universidade Federal do Rio de Janeiro, P. O. Box 68530, 21945-970 Rio de Janeiro, Brazil.}
\email{arbieto@im.ufrj.br}

\author[J. Bohorquez]{J. Bohorquez}
\address{Instituto de Ci\^encias Exatas e Aplicadas, Universidade Federal de Ouro Preto, Jo\~ao Monlevade, Brazil.}
\email{jennyffer.barrera@ufop.edu.br}

\usepackage{amsmath}
\usepackage{graphicx,color}
\newtheorem{theorem}{Theorem}

\newtheorem{lemma}[theorem]{Lemma}
\newtheorem{proposition}[theorem]{Proposition}

\theoremstyle{definition}
\newtheorem{definition}[theorem]{Definition}
\newtheorem{remark}[theorem]{Remark}
\newtheorem*{notation}{Claim}
\newtheorem{question}{Question}

\newcommand{\eps}{\varepsilon}

\newenvironment{ex}{{\bf Example: }}{\hfill\mbox{$\Diamond$}\\ }

\newcommand{\ack}{{\bf Acknowledgements. }}

\begin{document}

\begin{abstract}
Let $f: M \rightarrow M$ be a Morse-Smale diffeomorphism defined on a compact and connected manifold without boundary. Let $C(M)$ denote the hyperspace of all subcontinua of $M$ endowed with the Hausdorff metric and $C(f): C(M) \rightarrow C(M)$ denote the induced homeomorphism of $f$. We show in this paper that if $M$ is the unit circle $S^1$ then the induced map $C(f)$ has not the shadowing property. Also we show that the topological entropy of $C(f)$ has only two possible values: $0$ or $\infty$. In particular, we show that the entropy of $C(f)$ is $0$ when $M$ is the unit circle $S^1$ and it is $\infty$ if the dimension of the manifold $M$ is greater than two. Furthermore, we study the recurrence of the induced maps $2^f$ and $C(f)$ and sufficient conditions to obtain infinite topological entropy in the hyperspace.
\end{abstract}

\maketitle

\section{Introduction}

\label{d.sets}

Given a compact metric space $(X,d)$, $2^X$ is the hyperspace of nonempty closed subsets of $X$ with the Hausdorff metric $H_d$. A continuous map $f: X \rightarrow X$ induces a continuous map $2^f: 2^X \rightarrow 2^x$ defined by $2^f(A)=f(A)$ for all $A \in 2^X$. It is natural to ask which of the topological properties of $f$ carry over to $2^f$ and conversely. Such relationships has been studied during the last forty five years by several authors, see for instance  \cite{AIM} \cite{BS}  \cite{Nilson-romulo} \cite{FG} \cite{Paloma-Mendez} \cite{DO} \cite{LR}. When $(X,d)$ is a continuum space, that is, a nonempty compact and connected metric space then $f$ induces the map $C(f)= 2^f|_{C(X)}$ where $C(X)$ is the hyperspace of subcontinua of $X$. In this article we call $C(f)$ the induced continuum map of $f$. Many techniques used to study the map $2^f$ can not be applied to $C(f)$ like approximate a compact set by a finite number of points, see \cite{AIM} \cite{LR}. Actually, as we will show in this paper, some results for the continuum map $C(f)$ are different that the ones for $2^f$ and another one remain the same.\\

%

On the other hand, let $M^n$ is a $n$-dimensional compact and connected manifold without boundary. A diffeomorphism $f: M ^n \rightarrow M^n$ is called a Morse-Smale diffeomorphism if its nonwandering set $\Omega(f)$ consists of finitely many finite hyperbolic periodic points ($\Omega(f)=Per(f)$) whose invariant manifolds have mutually transversal intersections. In this paper, we are interested mainly in to study some topological properties of the induced continuum map $C(f)$ of Morse-Smale diffeomorphisms as: recurrence, shadowing property and topological entropy. We will explain bellow.\\

In Section \ref{Section-2}, we define the objects and state the results that will be needed throughout this article.\\

In Section \ref{Section-3}, we study the recurrence of the induced maps $2^f$ and $C(f)$ of Morse-Smale diffeomorphisms. We show that a finite number of periodic points on the base space could be generate a infinite set of periodic points on the hyperspace and which is dense in the non-wandering set. Besides that, with extra hypothesis we show that there exist uncountably many homoclinic points in the non-wandering set of $C(f)$.\\

Section \ref{Section-4} is devoted to study the shadowing property of the induced map of Morse-Smale diffeomorphisms. The notion of shadowing was introduced independently by Anosov \cite{Anosov} and Bowen \cite{Bowen} in the 1970s. It is a well known property in the qualitative theory of dynamical systems, see \cite{Aoki-Hiraide} \cite{pilyugim-shad}. 
Good and Fern\'andez, in \cite{FG}, showed that a map $f: X \rightarrow X$ on a compact metric space has the shadowing property if and only if its induced map $2^f$ has the shadowing property. Besides that, with the same arguments they showed that if $C(f)$ has the shadowing property then $f$ has shadowing. Nevertheless, it was not known if the conversely is true but we show that the conversely is false as follow: \\

\textbf{Theorem A.}
Let $f: S^1 \rightarrow S^1$ be a Morse-Smale diffeomorphism. Then the continuum map $C(f)$ does not have the shadowing property.\\

It is well known that the Morse-Smale diffeomorphisms on $S^1$ have the shadowing property. In fact, Yano \cite{Yano-2} showed necessary and sufficient conditions under which a homeomorphism of $S^1$ has the shadowing property and, in particular, the Morse-Smale diffeomorphisms satisfy this conditions. Also, the set of Morse-Smale diffeomorphisms is an open and dense set in Diff$\,^r(S^1)$, with $r \geq 1$. Thus, by Theorem A we have that there is an open and dense set in Diff$\,^r(S^1)$, such that $f$ has the shadowing property but the induced map $C(f)$ does not have shadowing. Besides that, it was shown that structurally stable systems have the shadowing property \cite{Robinson}  \cite{pilyugim-shad} \cite{pilyugim-stable}. In particular, Morse-Smale diffeomorphisms on manifolds with dimension $n \geq 2$ have the shadowing Property. We show that, with extra hypothesis, the induced map $C(f)$ does not have Shadowing.\\

And finally, in Section \ref{Section-5}, we study the topological entropy of the induced continuum map of Morse-Smale diffeomorphisms. Topological entropy is a way of measuring the complexity of a dynamical system. In this context, a dynamical systems is said to be chaotic if the topological entropy is positive. In \cite{BS}, Bauer and Sigmund showed that if $f$ is a continuous map and it has positive topological entropy then the induced map $2^f$ has infinite topological entropy. Following the same argument in \cite{BS}, if $(X,d)$ is a continuum and $f$ has positive topological entropy then the induced continuum map $C(f)$ has positive topological entropy, but not necessarily infinite, see examples in \cite{DO} \cite{LR}. However, the idea that  homoclinic orbit and heteroclinic orbit produces $\log 2$ in the hyperspace and so, zero entropy can induce positive entropy, appeared for the first time probably in \cite{DO}. The authors there also study some continuum dynamics of graph maps. On the other hand, Koichi Yano \cite{KY} showed that generic homeomorphisms on manifolds $n$-dimensional with $n\geq2$ has infinite topological entropy, therefore the induced hyperspace maps $2^f$ and $C(f)$ have infinite topological entropy. A natural question is: what happen if $f$ has topological entropy zero? \\

In the literature, there are some results that answer partially this question. For example, Lampart-Raith characterized the topological entropy of the induced continuum map of homeomorphisms defined on the unit circle $S^1$ or on the interval $I$, see \cite{LR}. They showed the topological entropy of the induced hyperspace map from homeomorphisms on the unit circle (or interval) can be zero or infinite, while the topological entropy of the induced continuum map is zero. Another example is the dendrite homeomorphism. In fact, the authors in \cite{AEO} proved that every homeomorphism defined on dendrite has zero entropy. And later, Acosta, Illanes and M\'endez-Lango in \cite{AIM} gave two examples of dynamical systems defined on different dendrites such that the continuum induced map has infinite topological entropy.\\

The Morse-Smale diffeomorphisms is a class of dynamical systems with zero topological entropy. And we show that the topological entropy of the induced map $C(f)$ could be zero or infinite as following\\

\textbf{Theorem B.} Let $f: M^n \rightarrow M^n$ be a Morse-Smale diffeomorphism, then the topological entropy of its induced continuum map $C(f)$ is zero or infinite.\\

Note that he topological entropy of the induced map $C(f)$ depends of the dimension of the base space. In particular, the entropy is $0$ when $M^n$ is the \textcolor{blue}{unit} circle and it is infinite if the dimension of the manifold is greater than one.


Also, we will give sufficient conditions to obtain infinite topological entropy to the induced maps $2^f$ and $C(f)$.

\section{Preliminaries}\label{Section-2}

In this section we define the objects that appear in the statements of the results in the introduction and collect some results which will be needed later.

\subsection{Morse-Smale diffeomorphisms}
\label{section-preliminares}

Let Diff$\,^r(M^n)$ denote the group of $C^r$ diffeomorphisms of a compact and connected $n$-dimensional manifold $M$ without boundary endowed with the $C^r$ topology, $r\geq 1$. For $f \in$ Diff$\,^r(M^n)$ we denote by $Per(f)$ the set of periodic points, by $Fix(f)$ the set of fix points, by $\Omega(f)$ the set of non-wandering points and by $\mathcal{R}(f)$ the set of recurrent points of $f$. A point $x \in Per(f)$ of period $n$ is hyperbolic if the derivative $(Df^n)_x$ has its spectrum disjoint from the \textcolor{blue}{unit} circle in the complex plane $\mathbb{C}$. In this case we have the existence of stable and unstable manifolds of $x$, denoted by $W^s(x)$ and $W^u(x)$.


\begin{definition}
$f \in$ Diff$\,^r(M^n)$ is called Morse-Smale if it satisfies the following conditions
\begin{itemize}
\item [i.] $\Omega(f)$ consists of a finite number of fixed and periodic points, all hyperbolic;
\item [ii.] the stable and unstable manifolds of the fixed and periodic points are all transversal to each other.
\end{itemize}
\end{definition}

If $M^n$ is the circle, then the following are true:
\begin{itemize}
\item The expanding and contracting	periodic points alternate.

\item If $f$ is orientation preserving, all periodic points have the same period, and if $f$ is orientation reversing all periodic points	have period one or two.

\item The set of Morse-Smale diffeomorphisms is dense in Diff$\,^r(S^1)$, $r\geq 1$.
\end{itemize}

Other important facts about Morse-Smale diffeomorphisms are the following:

\begin{itemize}
\item Every Morse-Smale diffeomorphism admits a filtration.

\item Every Morse-Smale diffeomorphism has an attractor periodic point and a repeller periodic point.

\item If $p$ is an attractor periodic point then exist $q$ a repeller periodic points such that the stable and unstable manifolds of these points intersects transversely.

\item For every Morse-Smale diffeomorphism, $M= \bigcup W^s(p_i)= \bigcup W^u(p_i)$ with $p_i \in \Omega(f)$. Therefore, there exist an attractor periodic point $p$ and a repeller periodic point  $q$ such that $W^s(p) \cap W^u(q)$ is a nonempty open set Diff$\,^r(M^n)$.

\item The set of Morse-Smale diffeomorphisms is open (and nonempty) in Diff$\,^r(M^n)$ for any manifold $M^n$ and any $r \geq 1$.

\item If $f \in$ Diff$\,^r(M^n)$ is Morse-Smale then $f$ is structurally stable.

\item The set of Morse-Smale diffeomorphisms is not dense in Diff$\,^r(M^n)$, $n \geq 2$.
\end{itemize}

\subsection{Hyperspace}

Let $X$ be a compact metric space, we consider the following hyperspaces of $X$:
\begin{itemize}
	\item $2^X= \{A \subset X: \textrm{$A$ is a nonempty and closed}\}$ is the \emph{hyperspace} of closed nonempty subsets of $X$.
	
    \item $C(X)= \{ A \in 2^X: \textrm{$A$ is connected}\}$
     is the \emph{continuum hyperspace} of $X$.

    \item $C_n(X)= \{ A \in 2^X: \textrm{$A$ has at most $n$ components}\}$
    is the \emph{n-fold hyperspace} of $X$.
 
    \item $F_n(X)= \{ A \in 2^X: \textrm{$A$ has at most $n$ points}\}$
    is the \emph{n-fold symmetric product hyperspace} of $X$.

\end{itemize}

\begin{definition}
Given a map $f:X \to X$ defined on compact metric space, the induced maps are given in the following way:
\begin{itemize}
\item The \emph{induced hyperspace map} $2^f: 2^X \rightarrow 2^X$ is given by $2^f(A)=f(A)$.

\item The \emph{induced continumm map} $C(f): C(X) \rightarrow C(X)$ is given by $C(f)=2^f|_{C(X)}$.

\end{itemize}
\end{definition}

Given a compact metric space $X$ with metric $d$. For any $r>0$ and any $A \in 2^X$, the \emph{open ball about} $A$ \emph{of radius} $r$ is given by:
$$ V(A,r)=\{x \in X: d(x,A) < r \}.$$

Since that $X$ is a compact metric space with metric $d$, then (see for example \cite{SM}) $2^X$ is a compact metric space when equipped with the Hausdorff metric

$$d_H(A,B)=\inf \{\varepsilon > 0: A \subset V(B,\varepsilon)\\\ \textrm{    and    } \\\ B \subset V(A,\varepsilon)\}.$$

The topology generated by $d_H$ coincides with the Vietoris topology. Given a non-empty compact set $A \in 2^X$ and $r>0$, we define the \emph{open ball of center} $A$ and radius $r$ as the set
$$B_H(A,r)=\{K \in 2^X: d_H(K,A)<r\}.$$

\begin{remark}
If $f$ is a homeomorphism, the induced maps $2^f$ and $C(f)$ are also homeomorphisms. Moreover, if $f$ and $g$ are topologically conjugate then $2^f$ (or $C(f)$) and $2^g$ (or $C(g)$) are topologically conjugate.
\end{remark}

A \emph{continuum} is a nonempty, compact and connected metric space. The topological limit, with respect to the Hausdorff metric, of a sequence of closed nonempty sets $(K_n)_n$ in a metric space is denoted by Lim$K_n$.


\subsection{Shadowing}
Given a homeomorphism $f: X \rightarrow X$ on a compact metric space $X$, recall that a sequence $(x_n)_{n \in Z}$ is called a $\delta$-\emph{pseudo orbit}  of $f$ if
$$
d(f(x_n),x_{n+1}) \leq \delta \hspace{0.2cm} \textrm{for all} \hspace{0.2cm} n \in \mathbb{Z}
$$
The homeomorphism $f$ is said to have the \emph{shadowing property} if for every $\varepsilon >0$ there exists $\delta>0$ such that every $\delta$-pseudo orbit $\{x_n\}_{n \in \mathbb{Z}}$ of $f$ is $\varepsilon$-shadowed by a real orbit of $f$; i.e, there exists $x \in X$ such that 
$$
d(x_n,f^n(x)) < \varepsilon  \hspace{0.2cm} \textrm{for all} \hspace{0.2cm} n \in \mathbb{Z}.
$$

It is easy to show that $f$ has the shadowing property if and only if $f^n$ has the same property (with other constants). See Aoki and Hiraide \cite{Aoki-Hiraide} (1999, p. 79).  


Examples of dynamical systems with the shadowing property are the Morse-Smale diffeomorphisms defined on $S^1$. See Pilyugin \cite{pilyugim-shad} (1999,p. 174).

\subsection{Topological entropy}
We give the definition of topological entropy using separating sets. This was done by Dinaburg and Bowen, see \cite{w}. Let $(X,d)$ be a compact metric space and $f: X \rightarrow X$ be a continuous map. If $n$ is a natural number then $d_n(x,y)= \max_{0 \leq i \leq n-1} d(f^i(x), f^i(y)$. 

Let $n$ be a natural number and $\varepsilon >0$. A subset $E$ of $X$ is said to be $(n,\varepsilon)$-\emph{separated} if $x, y \in E$, $x \neq y$, implies $d_n(x,y)>\varepsilon$ (i.e., for every point $x \in E$ the set $\bigcap_{i=0}^{n-1}f^{-i}B(f^i(x),\varepsilon)$ contains no other point of $E$). Let $s(n,\varepsilon)$ be the largest cardinality of any $(n,\varepsilon)$ separated subset of $X$. Then the topological entropy of $f$ is given by
$$
h(f)= \lim_{\varepsilon \rightarrow 0} \limsup_{n \rightarrow \infty} \frac{1}{n} \log s(n,\varepsilon).
$$


\section{Recurrence of the induced Morse-Smale diffeomorphisms}\label{Section-3}

Morse-Smale diffeomorphisms could generate interesting topological properties in its induced hyperspace maps $2^f$ and $C(f)$, as we will show in this section. Some of this properties has been well studied in the last years for continuous maps, see for example \cite{BS} \cite{Nilson-romulo} \cite{DO}, but we will provide them here. To prove Proposition \ref{dyn-circle-2^f} we follow the techniques used in \cite{Nilson-romulo} by N. Bernardes and R. Vermersch about \textit{dumbbells}. Note that in the proof of the following proposition we use the information of the dynamics of Morse-Smale diffeomorphisms on the circle, that is, the expanding and contracting periodic points alternate. 

\begin{proposition}
\label{dyn-circle-2^f}
Let $f:S^1\to S^1$ be a Morse-Smale diffeomorphism then there exists $n \in \mathbb{N}$ such that
\begin{itemize}
\item [i.] $2^{f}$ has uncountably many periodic points of each period $k\geq 1$;
\item [ii.] $R(2^{f}) \neq \Omega(2^{f}) = \overline{Per(2^{f})}= CR(2^{f})$;
\item [iii.] The induced map $2^f$ is not transitive;
\end{itemize}
where $R(2^f)$ and $CR(2^f)$ denote, respectively, the set of recurrent and chain recurrent points of $2^f$.
\end{proposition}

\begin{proof}

\begin{itemize}
\item [i.] Fix an arbitrary number $k \in \mathbb{N}$. For all $x \in S^1 \setminus Per(f)$, we have the set $\Lambda_k=\overline{ \cup_{i \in \mathbb{Z}}\{f^{ki}(x)\}}$ is a $k$-periodic point of $2^f$.
\item [ii.] Without loss of generality, we can assume that the diffeomorphism $f$ preserves orientation and therefore all periodic points have the same period (if $f$ is orientation reversing, the proof is analogous). Set $L$ the period. For $x, y \in S^1$ let $[x,y]$ be all points of $S^1$ lying in clockwise sense between $x$ and $y$ ($[x,y]=\{x\}$ if $x=y$). Let $Per(f)= \{p_1, p_2, ..., p_{j_0}\}$ where $p_{2i-1}$ are attractor points, $p_{2i}$ are repeller points and satisfy the following conditions: for $i=1,2,...,j_0-2$ 
$$W^s(p_i) \cap W^u(p_{i+1}):= W^s(f^L(p_i)) \cap W^u(f^L(p_{i+1}))=(p_i,p_{i+1})$$and 
$$W^u(p_{i+1}) \cap W^s(p_{i+2}):= W^u(f^L(p_{i+1})) \cap W^s(f^L(p_{i+2}))=(p_{i+1},p_{i+2}),$$
and for $i=j_0-1$
$$W^s(p_{j_0-1}) \cap W^u(p_{j_0}): = W^s(f^L(p_{j_0-1})) \cap W^u(f^L(p_{j_0}))= (p_{j_0-1},p_{j_0})$$
and
$$W^u(p_{j_0}) \cap W^s(p_1):= W^u(f^L(p_{j_0})) \cap W^s(f^L(p_1)) =(p_{j_0},p_1).$$
 
First we show that there exists $K \in \Omega(2^f)$ but $K \notin R(2^f)$. For this let $x \in S^1 \setminus Per(f)$ and define
$$
K:= \{p_1,p_2,...,p_{j_0}\} \cup \{x\} \in 2^{S^1}.
$$
Since that $K$ is a homoclinic point of $2^f$, i.e., $$\lim\limits_{n \rightarrow + \infty}f^n(K)= Per(f)=\lim\limits_{n \rightarrow - \infty} f^{n}(K)$$
then there exists $\varepsilon>0$ such that $d_H(K,{2^f}^n(K)) > \varepsilon >0$ for all $n \in \mathbb{N}$ and so we have that $K$ is not a recurrent point of $2^f$. On the other hand, for each $n \in \mathbb{N}$, let 
$$
K_n:= K \cup \{f^{-n}(x)\} \in 2^{S^1}.
$$
Then, we obtain 
$$
\lim\limits_{n \rightarrow \infty}K_n = K \hspace{0.2cm} \textrm{and} \hspace{0.2cm} \lim\limits_{n \rightarrow \infty}f^{n}(K_n) = K,
$$
that is, $K$ is a non-wandering point of $2^f$.

Now, we will show the first equality. For this, first we claim if $K \in  \Omega(2^f)$ and $K \cap (p_i,p_{i+1}) \neq \emptyset$ then $p_i, p_{i+1} \in K$. For this, we will suppose that $p_i$ is an attractor and $p_{i+1}$ is a repeller. The proof in the other case is analogous.

In fact, suppose that there exists $x \in K \cap (p_i,p_{i+1}) $. Since that $K \in \Omega(2^f)$, there exist two sequences $\{A_m\}_{m \in \mathbb{N}}$ and $\{f^{N_m}(A_m)\}_{m\in \mathbb{N}}$  in $2^{S^1}$ with $N_m < N_{m+1}$ such that

\begin{equation}
\lim_{m \rightarrow \infty} A_m= K \hspace{0.3cm} \textrm{and}  \hspace{0.3cm}  \lim_{m \rightarrow \infty} f^{N_m} (A_m)= K .
\end{equation}
Therefore, there exist two sequences  $\{a_m\}_{m \geq 1}$ and $\{f^{N_m}(a_m)\}_{m \geq 1}$ in $S^1$ with the following properties: $\lim_{m \rightarrow \infty} a_m=x$  and  the accumulation points of $\{f^{N_m} (a_m)\}_{m \geq 1}$ is contained in $K$. Without loss of generality, we can assume that there exits a sequence $t_m \geq 1$ satisfying $N_m=Lt_m+j$ with $0 \leq j \leq L-1$ and

$$\lim_{m \rightarrow \infty} f^{N_m}(a_m)= \lim_{m \rightarrow \infty} f^{Lt_m+j}(a_m) \in K .$$

Since $a_m \in (p_i, p_{i+1})$ then $f^j(a_m) \in (f^j(p_i), f^j(p_{i+1}))$ and 

$$f^{Lt_m+j}(a_m) \in (f^{Lt_m}(f^j(p_i)), f^{Lt_m}(f^j(p_{i+1})))= (f^j(p_i), f^j(p_{i+1}))$$

for all $m \geq m_0$.

\begin{notation}
$$\lim_{m \rightarrow \infty} f^{N_m}(a_m)= f^j(p_i) \in K .$$ 
\end{notation}

\begin{proof}
Given $\varepsilon>0$, there exists $T>0$ such that $d(f^{LT+j}(x),f^j(p_i))<\varepsilon/2$ and, using the invariance of the stable manifold,  for any $t>T$ then $d(f^{Lt+j}(x),f^j(p_i))<\varepsilon/2$. By uniform continuity, there exists $\delta>0$ such that if $d(y,x)<\delta$ then $d(f^i(x),f^i(y))<\varepsilon/2$ if $i=0,\dots, LT+j$. Therefore, for such $y$ $d(f^{Lt+j}(y),f^j(p_i))<\varepsilon$ if $t>T$.

Now, let $m_0$ such that if $m>m_0$ then $d(x,a_m)<\delta$. Since $N_m\to \infty$ then there exists $m_1>m_0$ such that $N_m=Lt+j\geq LT+j$ if $m>m_1$. Thus, if $m>m_1$ then $d(f^{N_m}(a_m),f^ j(p_i))=d(f^{Lt+j}(a_m),f^j(p_i))<\varepsilon$. The claim follows.
\end{proof}



If $j=0$ then we are done. Otherwise, as $f^j(p_i) \in K$ by $(1)$ we have there exists a sequence $\hat{a}_m \in A_m$ such that $\lim_{m \rightarrow \infty} \hat{a}_m= f^j(p_i)$ and the accumulation points of $\{f^{N_m} (\hat{a}_m)\}_{m \geq 1}$ is contained in $K$. 

\begin{notation}
$$\lim_{m \rightarrow \infty} f^{N_m}(\hat{a}_m)= f^{2j}(p_i) \in K .$$ 
\end{notation}

\begin{proof}
Given $\varepsilon>0$, using the invariance of the stable manifold of $f^j(p_i)$, we know that for $t\geq 0$, we have
$$f^{Lt+j}((f^j(p_i)-\varepsilon,f^j(p_i)+\varepsilon))\subset (f^{2j}(p_i)-\varepsilon ,f^{2j}(p_i)-\varepsilon).$$
There exists $m_0$ such that if $m>m_0$ then $\hat{a}_m\in (f^j(p_i)-\varepsilon,f^j(p_i)+\varepsilon)$. Hence, if $m>m_0$ then
$$f^{N_m}(\hat{a}_m)=f^{Lt_m+j}(\hat{a}_m)\in (f^{2j}(p_i)-\varepsilon , f^{2j}(p_i)-\varepsilon).$$
\end{proof}







Following the same argument we have that
$$\{f^j(p_i), f^{2j}(p_i),..., f^{Lj}(p_i)=p_i \} \subset K.$$ 
To prove that $p_{i+1} \in K$, we use the same argument but with $m$ negative.

Consider $K \in \Omega(2^f)$ and $\varepsilon >0$. Let $\{l_1,\dots,l_s\}=\{i;K\cap (p_i,p_{i+1})\neq\emptyset\}$. For each nonempty set $K \cap (p_i,p_{i+1})$ there exist $N_i$ elements $x_r^i$ satisfying $K \cap (p_i,p_{i+1}) \subset \cup_{r \leq N_i} B(x_r^i, \varepsilon)$ with $K \cap B(x_r^i, \varepsilon) \neq \emptyset$. Thus we obtain a finite set
$$
X= \{x_1^{l_1},...,x_{N_1}^{l_1} \} \cup ... \cup \{x_1^{l_s},...,x_{N_{l_s}}^{l_s} \}
$$
and $N \geq 1$ such that $f^{Ln}(x), f^{-Ln}(x) \in \bigcup_{i \in \{l_1,...,l_s\}} B(p_i, \varepsilon)$ for all $x \in X$ and $n \geq N$. Finally, we consider $Z= \overline{\cup_{j \in \mathbb{Z}} f^{jLN}(X)} \cup [K \cap Per(f)]$. Thus
$$
f^{LN}(Z)= Z \hspace{0.2cm} \textrm{and} \hspace{0.2cm} d_H(Z, K) < \varepsilon.
$$
Thus, we have the first equality $\Omega(2^f)=\overline{Per(2^f)}$. If $f$ is orientation reversing the proof is analogous. The second equality follows from Theorem $6$ in \cite{FG} and Theorem 3.1.2 in \cite{Aoki-Hiraide}.
\item [iii] Follows from  Theorem 2.1 in \cite{Peris}.
\end{itemize}
\end{proof}

The dynamics of the induce map $C(f)$ is completely different than the induced map $2^f$ as we will show. In fact, we show that the non-wandering set of its induced continuum map $C(f)$ is a finite number of periodic points. The Lemma \ref{Theorem-nowandering-finite} will be used to prove Theorem B in Section \ref{Section-5}.

\begin{lemma}
\label{Theorem-nowandering-finite}
Let $f: S^1 \rightarrow S^1$ be a Morse-Smale diffeomorphism then
\begin{itemize}
\item [i.] If $f$ preserves orientation and all periodic points of $f$ have period $N$, then $$\Omega(C(f))= Per_N(C(f)) \cup \{S^1\};$$
\item [ii.] If $f$ reverses orientation, then $$\Omega(C(f))= Per_2(C(f)) \cup Fix(C(f)).$$
\end{itemize}
\end{lemma}

\begin{proof}
	
\begin{itemize}
		
\item[i.] Let $f: S^1 \rightarrow S^1$ be orientation preserving Morse-Smale diffeomorphism such that every point in the non-wandering set is a fixed point, i.e, $\Omega(f)=Fix(f)=\{p_1,\dots,p_k\}$.  We say that a subset $J\subset S^1$ is a $D$-arc if it is homeomorphic to an interval (possibly degenerated) and the boundary of $J$ is contained in the set of fixed points, which we denote by $\partial J\subset Fix(f)$. Note that, if $f$ preserves orientation and $J \in C(S^1)$ then $J=[a,b]$ for some $a,b \in S^1$ and $f([a,b])=[f(a),f(b)]$. Therefore, we have that $Fix(C(f))= \{J:  \textrm{$J$ is a $D$- arc}\} \cup \{S^1\}$.
		
We will show that $\Omega(C(f))= Fix(C(f))$, i.e, the non-wandering set consists of finitely many fixed points. Indeed, suppose that there is an interval $I=[x,y]$ in $\Omega(C(f))$ such that $I$ is not a fixed point of $C(f)$. Without loss of generality we can suppose that $x$ is not a fixed point of $f$ and therefore there is $p_i \in Fix(f)$ such that $x \in W^s(p_i)$ (see definition in \ref{section-preliminares}). Since that $I \in \Omega(C(f))$, we have that for every  $n \in \mathbb{N}$, there are $N_n \geq 1$ and an interval $J_n=[a_n,b_n]$ such that 
$$d_H(I,J_n)< \frac{1}{n} \\\  \\\ \textrm{   and   } \\\ \\\ d_H(I, C(f)^{N_n}(J_n)) < \frac{1}{n}.$$ 
This implies that $\lim\limits_{n \rightarrow \infty} a_n = x$ and $\lim\limits_{n \rightarrow \infty} f^{N_n}(a_n)= x$. Since $f$ is a continuous map, for $\varepsilon =\min\{d(x,f(x)), d(f(x),p_i)\}/2$, there is $ \delta>0$ small enough such that $f(B(x, \delta))\subset B(f(x),\varepsilon)$ and $B(x,\delta) \cap B(f(x), \varepsilon) = \emptyset$. Besides that, there is $n_0 \geq 1$ such that $a_{n_0}$ and $f^{N_{n_0}}(a_{n_0})$ belong to $B(x, \delta)$. As $f$ preserves orientation we have that $f^n(a_{n_0}) \in [p_i,f(x)] \cup B(f(x),\varepsilon)$ for every $n \geq 1$ which is a contradiction. This contradiction yields the proof.
		
Now, suppose that $f$ is orientation preserving and the non-wandering set consists of a finite number of periodic points with the same period $N>1$. In this case we say that $J$ is a $D$-arc of $f^N$ if and only if $J$ is homeomorphic to an interval (possibly degenerated) and $\partial J \subset Fix(f^N)$. Therefore, $Per_N(C(f))=\{J: \textrm{$J$ is a $D$-arc of $f^N$}\}$. As above, following the same arguments, we can prove that $\Omega(C(f))= \{S^1\} \cup \{J: \textrm{$J$ is a $D$-arc of $f^N$}\}$, i.e., the set of non-wandering points of its induced continuum map $C(f)$ consists of finitely many periodic points (all with the same period $N$) and a fixed point. 

\item[ii.] Suppose that $f$ is orientation reversing, then all periodic points have period one or two. In this case we introduce three kinds of arcs: $D$-arc, $D$-mix-arc and $D^*$-mix-arc. We defined $D$-arc above in item i. We say that $J \subset S^1$ is a $D^*$-mix-arc if it is homeomorphic to an interval and $\partial J=\{p,f(p)\}$ for some $p \in Per_2(f)$. And, we say that $J \subset S^1$ is a $D$-mix-arc if it is homeomorphic to an interval (possibly degenerated) and $\partial J=\{p_1,p_2\} \subset Per(f)$ with $\mathcal{O}(p_1) \neq \mathcal{O}(p_2)$. Therefore, we have that 
$$
Fix(C(f))= \{ \textrm{degenerated D-arc}\} \cup \{ \textrm{$D^*$-mix-arc}\} \cup \{S^1\}
$$
and 
$$
Per_2(C(f))=\{ \textrm{no degenerated D-arc}\} \cup \{ \textrm{D-mix-arc}\}. 
$$
As item i. we can prove that $\Omega(C(f))= Fix(C(f)) \bigcup Per_2(C(f))$. 
\end{itemize}
\end{proof}

We do not have enough tools to show item ii. of Proposition \ref{dyn-circle-2^f} for induced map $C(f)$ of Morse-Smale diffeomorphisms defined on manifolds $M^n$, with $n \geq 2$. Indeed, in the last years the problem of topological classification of this group of diffeomorphisms continues to be object to studied, see \cite{Bonnati-GMP}, \cite{G-Z-M} and \cite{Z-M}. However, we show that if $f$ is a North Pole-South Pole diffeomorphism, then item ii. holds. The proof we present to Proposition \ref{t.PNPS} follows, with slight changes, the proof given for Proposition \ref{dyn-circle-2^f}. The main difference is the construction of the continuum in the non-wandering set of the induced map $C(f)$.


 
\begin{proposition}
\label{t.PNPS}
If $f:M^n\to M^n$ is a Morse-Smale diffeomorphism with only two fixed points. Then the following properties hold:
\begin{itemize}
\item [i.]$C(f)$ has uncountably many periodic points of each period $N\geq 1$;
\item [ii.] C(f) is not transitive;
\item [iii.]$R(C(f)) \neq \Omega(C(f))$;
\item [iv.] $\Omega(C(f))=\overline{Per(C(f))}$;
\end{itemize}
\end{proposition}

\begin{proof}
Suppose that $p$ is an attracting fixed point and $q$ is a repeller fixed point and $M=S^2$.
\begin{itemize}
\item [i.]  Fix $N \geq 1$ and $x \in M \setminus \{p,q\}$. We consider an arc $\gamma$ between $x$ and $f^N(x)$ such that $\gamma \cap \{f^j(x): j= \pm 1, \pm 2, \ldots, \pm(n-1) \}= \emptyset$. So
$$
K= \overline{\bigcup_{j \in \mathbb{Z}} f^{jN}(\gamma)} \in C(M)
$$
is a $N$-periodic point of $C(f)$. 
\item [ii.] Follows of Theorem $3.1$ in \cite{fedeli}.
\item [iii.] Let $x \in M \setminus \{p,q\}$ and one arc $\gamma$ between $x$ and $f(x)$. Consider the fixed point 
$$
P= \overline{\bigcup_{j \in \mathbb{Z}}f^j(\gamma)}
$$
and note that $p,q \in P$. Besides that, we consider an arc $\beta$ such that one of the end point is $x$, $P \cap \beta =\{x\}$ and $f^n(\beta) \cap f^m(\beta) = \emptyset$ for all $n \neq m \in \mathbb{Z}$. Set
$$
K= P \cup \beta.
$$
Since $\lim_{n \rightarrow \pm \infty} C(f)^n(K)= P$ then $K$ is a homoclinic point and therefore there is $\eta>0$ such that
$$
d_H(K,C(f)^n(K)) \geq \eta \hspace{0.2cm} \textrm{for all} \hspace{0.2cm} n \in \mathbb{Z} .
$$
Thus $K$ is not a recurrent point.
On the other hand, for each $n \in \mathbb{N}$, let 
$$
K_n= K \cup \{f^{-n}(\beta)\} \in C(S^2).
$$
Then, we obtain 
$$
K_n \rightarrow K \hspace{0.2cm} \textrm{and} \hspace{0.2cm} C(f)^{n}(K_n) \rightarrow K\hspace{0.2cm} \textrm{when} \hspace{0.2cm} n \rightarrow \infty.
$$
Thus, $K$ is a non-wandering point of $C(f)$.

\item [iv.] First, note that if $K \in \Omega(C(f))$ is a non-degenerate compact and connected set, then $p,q \in K$. In fact, suppose that $p \not \in K$. Then there exists $\varepsilon>0$ such that $ p \not \in \overline{V(K, \varepsilon)}$. Since $\overline{V(K, \varepsilon)}$ is a compact and connected set we have that for $\delta>0$ small enough there exists $N_0 \in \mathbb{N}$ such that 
$$C(f)^{-N}(\overline{V(K, \varepsilon)}) \subset B_H(\{q\}, \delta) \hspace{0,2cm} \textrm{for } \hspace{0,2cm} n \geq N_0.$$  
Let $\varepsilon_1 < \min \{\varepsilon, \delta\}$ small enough such that  
$$B_H(f^{i}(K), \varepsilon_1) \cap B_H(f^{j}(K), \varepsilon_1)= \emptyset$$ for $i \neq j \in \{0,-1, \ldots,-N_0 \}$ and 
$$B_H(f^{-N_0}(K), \varepsilon_1) \subset B_H(\{q\}, \delta).$$
Since $C(f)$ is continuous, there exists $\varepsilon_2< \varepsilon_1$ with 
$$f^{-N_0}(B_H(K, \varepsilon_2)) \subset B_H(f^{-N_0}(K), \varepsilon_1) \subset B_H(\{q\}, \delta).$$
By construction, $B_H(K, \varepsilon_2) \cap C(f)^{-n}(B_H(K, \varepsilon_2))= \emptyset$ for all $n \geq 1$ and that is a contradiction. The cases with $q \not \in K$ and $p,q \not \in K$ are analogous.

Now, let $K \in \Omega(C(f))$ and $\varepsilon>0$ small enough such that
$$
f(\overline{W^s_{\varepsilon}(p)}) \subset W^s_{\varepsilon}(p) \hspace{0.2cm} \textrm{and} \hspace{0.2cm} f^{-1}(\overline{W^u_{\varepsilon}(q)}) \subset W^u_{\varepsilon}(q). 
$$
As $W^s_{\varepsilon}(p)$ and $W^u_{\varepsilon}(q)$ are an open sets and $K$ is compact, then
$$
\hat{K}=K \cap   [ M \setminus W^s_{\varepsilon}(p) \cup W^u_{\varepsilon}(q) ]
$$
is a compact set. Therefore, there is $N \in \mathbb{N}$ such that
$$
f^n(\hat{K}) \subset W^s_{\varepsilon}(p) \hspace{0.2cm} \textrm{and} \hspace{0.2cm}  f^{-n}(\hat{K}) \subset W^u_{\varepsilon}(q)
$$ 
for all $n \geq N$. If $\hat{K}$ is a connected set, then consider two points, one point $x \in \partial W^s_{\varepsilon}(p) \cap \hat{K}$ and one point $y \in \partial W^u_{\varepsilon}(q) \cap \hat{K}$, and an arc $\gamma$ with end points $x$ and $f^N(y)$. Finally, we consider $P= \overline{\cup_{j \in \mathbb{Z}} f^{Nj}(\hat{K} \cup \gamma)}$. Thus
$$
f^N(P)= P \hspace{0.2cm} \textrm{and} \hspace{0.2cm} d_H(P, K) < \varepsilon.
$$
If $\hat{K}$ is not a connected set, then let $\hat{K}_{\lambda}$ be the connected components of $\hat{K}$. Then, in each component we can consider points $x_{\lambda}$, $y_{\lambda}$ and arcs $\gamma_{\lambda}$ as above. Thus, we consider the periodic point  
$$
P= \overline{\bigcup_{j \in \mathbb{Z}} f^{Nj} (\bigcup_{\lambda \in \Lambda} \hat{K}_{\lambda} \cup \gamma_\lambda)}.
$$ 
Therefore, we have the equality $\Omega(C(f))=\overline{Per(C(f))}$.

\end{itemize}
\end{proof}


In Proposition \ref{t.PNPS} we show that the induced map $C(f)$ of the North Pole-South Pole diffeomorphism has uncountable many homoclinic points in the non-wandering set and each of these elements is accumulated by periodic orbits of $C(f)$. Also, it has homoclinic points which are accumulated by homoclinic points with the same omega limit set. The Morse-Smale diffeomorphism on the torus with four fixed points is another example with the same property. In the following result we give conditions to obtain the same property to Morse-Smale diffeomorphisms:

\begin{proposition}
Let $f: M^n \rightarrow M^n$ ($n \geq 2$) be a Morse-Smale diffeomorphism with an attractor and repeller points $p,q$, and $ x \in W^s(p) \cap W^u(q) \neq \emptyset$. If there is an arc $\gamma \subset W^s(p) \cap W^u(q)$ such that the end points of $\gamma$ are $x$ and $f(x)$ then the induced continuum map $C(f)$ has uncountably many homoclinic points in $\Omega(C(f))$. 
\end{proposition}

\begin{proof}
Consider $P= \overline{\bigcup_{j \in \mathbb{Z}} f^j(\gamma)}$ the fixed point of $C(f)$. Then, for every nontrivial arc $\beta \subset W^s(p) \cap W^u(q)$ with end point $x$ and $\beta \cap P=\{x\}$ we have the following set 
$$
K = P \cup \beta
$$
is a homoclinic point of $C(f)$. In fact, 
$$\lim_{n \rightarrow \infty}C(f)^n(K)= \lim_{n \rightarrow -\infty}C(f)^n(K)=P.$$

To show that $K \in \Omega(C(f))$, consider the sequence $K_n= K \cup \{f^{-n}(\beta)\}$ and note that $\lim_{n \rightarrow \infty} K_n=K$ and $\lim_{n \rightarrow \infty} f^n(K_n)=K$ 
\end{proof}


\section{Shadowing of induced Morse-Smale diffeomorphisms}
\label{Section-4}
This section is devoted to study the shadowing property of induced maps $2^f$ and $C(f)$ of Morse-Smale diffeomorphisms defined on a $n$-dimensional compact and connected manifold without boundary. It is well known that this group of diffeomorphisms defined on the circle have the shadowing property, see \cite{Yano-2} and \cite{pilyugim-shad}. And if $n \ge 2$, then $f$ it is structurally stable and therefore it has the shadowing property too, see \cite{pilyugim-shad}. By the follow proposition we have that the induced map $2^f$ inherits the shadowing property:\\

\begin{proposition}
\label{shadow-hypers}
Let $f: M^n \rightarrow M^n$ be a Morse-Smale diffeomorphism then the induced map $2^f$ has the shadowing property.
\end{proposition}

\begin{proof}
If $n=1$, it follows from Theorem $6$ in \cite{FG} and Theorem $3.1.1$ in \cite{pilyugim-shad}. If $n \ge 2$, it follows from Theorem $6$ in \cite{FG} and Theorem $2.2.5$ and Theorem $2.2.6$ in \cite{pilyugim-shad}.
\end{proof}
Proposition \ref{shadow-hypers} does not hold for the induced map $C(f)$ of Morse-Smale diffeomorphisms on the circle. In fact, the following is the main result of this section.\\

\textbf{Theorem A.}
Let $f: S^1 \rightarrow S^1$ be a Morse-Smale diffeomorphism. Then the continuum map $C(f)$ does not have the shadowing property.

\begin{proof}
Without loss of generality, we can assume that all periodic points of $f$ are fixed (otherwise we pass to some iteration of $f$). First, suppose that $f$ has only two fixed points, an attractor $p$ and a repeller $q$. Furthermore, suppose that $C(f)$ has the shadowing property. Fix $\varepsilon>0$ such that $B_H(\{p\}, \varepsilon) \cap B_H(\{q\}, \varepsilon) \cap B_H(S^1, \varepsilon)= \emptyset$. Then there is $\delta>0$ such that every $\delta$-pseudo orbit $\{x_n\}$ of $C(f)$ is $\varepsilon$-shadowed by real orbit of $C(f)$ and for each $\delta$, we can consider the following $\delta$-pseudo orbit:
$$
x_0=S^1, x_1=J_p^*, x_{-1}=J_q^*
$$ 
$$
x_i=C(f)^i(x_1) , \hspace{0.2cm} \textrm{and} \hspace{0.2cm} x_{-i}=C(f)^{-i}(x_{-1}) \hspace{0.2cm} \textrm{if} \hspace{0.2cm} i \geq 2
$$
where $J_p^*$ is an arc which contains the point $p$ but not the point $q$ and such that $d_H(J_p^*,S^1) <\delta$. In the same way, $J_q^*$ is an arc that contains the point $q$ but not the point $p$ and such that $d_H(J_q^*,S^1) <\delta$. Observe that 
$$
\lim_{i \rightarrow \infty}C(f)^i(x_1)= \{p\} \hspace{0.2cm} \textrm{and} \hspace{0.2cm} \lim_{i \rightarrow - \infty}C(f)^i(x_{-1})= \{q\}. 
$$
Fix $\varepsilon>0$ small enough and $\delta>0$ of the shadowing property. We claim that the $\delta$-pseudo orbit $\{x_i\}$ defined above is not shadowed. Indeed, suppose that there is $A \in C(S^1)$ such that $d_H(C(f)^i(A),x_i) <\varepsilon$ for all $i \in \mathbb{Z}$. Then we have two possibilities, $A$ contains $q$ or $A$ does not contain $q$. If $A$ contains $q$ then
$$\lim_{i \rightarrow \infty } C(f)^i (A)= S^1 \hspace{0.2cm} \textrm{and} \hspace{0.2cm} \lim_{i \rightarrow - \infty}C(f)^i(A) \in  \{ \{q\}, S^1\}$$
and therefore there is $i_0>1$ big enough such that $d_H(C(f)^{i_0}(A),x_{i_0}) > \varepsilon$ and it is a contradiction. If $A$ does not contain $q$ then
$$\lim_{i \rightarrow \infty } C(f)^i (A)= \{p\} \hspace{0.2cm} \textrm{and} \hspace{0.2cm} \lim_{i \rightarrow - \infty}C(f)^i(A)=S^1$$
and therefore there is $j_0>1$ big enough  such that $d_H(C(f)^{-j_0}(A), x_{-j_0})> \varepsilon$ and it is a contradiction too. Therefore, $C(f)$ does not have the shadowing property. 

If $f$ has more than two fixed points, we consider two consecutive points, one attractor point and one repeller point and define a $\delta$-pseudo orbit as above. The difference is that $\lim_{i \rightarrow \infty}C(f)^{i}(x_1)$ and $\lim_{i \rightarrow \infty}C(f)^{i}(x_{-1})$ are $D$-arcs, i.e., arcs whose end points are fixed points of $f$ (see definition in following section).
\end{proof}

Theorem A remains open for Morse-Smale diffeomorphisms on manifolds with dimension $n\geq 2$. However, for North Pole - South Pole diffeomorphisms on $2$-sphere we generalize the proof of Theorem A. 

\begin{proposition}
If $f:S^2\to S^2$ is a Morse-Smale diffeomorphism with only two fixed points. Then the induced map $C(f)$ has not the shadowing property.
\end{proposition}

\begin{proof}
Suppose that $C(f)$ has the shadowing property. Let $S_1$ be the circle contains the points $p$ and $q$. Let $\varepsilon>0$ small enough such that 
$$B_H(\{p\}, \varepsilon) \cap B_H(\{q\}, \varepsilon) \cap B_H(S_1, \varepsilon)= \emptyset.$$
Then there is $\delta>0$ such that every $\delta$-pseudo orbit $\{x_n\}$ of $C(f)$ is $\varepsilon$-shadowed by real orbit of $C(f)$ and for each $\delta$, we can consider the following $\delta$-pseudo orbit:
$$
x_0=S_1, x_1=S_1 \setminus {B(q, \delta/2)}, x_{-1}=S_1\setminus {B(p, \delta/2)}
$$ 
$$
x_i=C(f)^i(x_1) , \hspace{0.2cm} \textrm{and} \hspace{0.2cm} x_{-i}=C(f)^{-i}(x_{-1}) \hspace{0.2cm} \textrm{if} \hspace{0.2cm} i \geq 2
$$
where $x_1=S_1 \setminus B(q, \delta)$ is an arc which contains the point $p$ but not the point $q$ and such that $d_H(x_1,S_1) <\delta$. In the same way, $x_{-1}=S_1\setminus B(p, \delta/2)$ is an arc that contains the point $q$ but not the point $p$ and such that $d_H(x_{-1},S_1) <\delta$. Observe that 
$$
\lim_{i \rightarrow \infty}C(f)^i(x_1)= \{p\} \hspace{0.2cm} \textrm{and} \hspace{0.2cm} \lim_{i \rightarrow - \infty}C(f)^i(x_{-1})= \{q\}. 
$$
Fix $\varepsilon>0$ small enough and $\delta>0$ of the shadowing property. We claim that the $\delta$-pseudo orbit $\{x_i\}$ defined above is not shadowed. Indeed, suppose that there is $A \in C(S^2)$ such that $d_H(C(f)^i(A),x_i) <\varepsilon$ for all $i \in \mathbb{Z}$. \\

\textbf{Claim 1:} The fixed points $p,q \not\in A$.\\
It follows from the proof of Theorem A\\

\textbf{Claim 2:} Let $S_2$ and $S_3$ circles containing the fixed points $p$ and $q$ and such that $S_i \cap S_j=\{p,q\}$ if $i \neq j \in \{1,2,3\}$. Then 
$$A \cap R^c\neq \emptyset$$
where $R$ is the region limited by $S_2$ and $S_3$ which contains $S_1-\{p,q\}$ in its interior. \\
In fact, suppose that $A\cap R^c=\emptyset$, then $A \subset R$. On the other hand, $R$ is a disconnected set with two connected components $R_1$ and $R_2$.  Since by hypothesis  $A$ is a continuum such that $d_H(A, S_1) < \varepsilon$ then $A \cap R_i \neq \emptyset$ for $i \in \{1,2\}$. Therefore $p$ or $q$ belongs to $A$, which is a  contradiction. \\

Let $S_2$ and $S_3$ as Claim 2 and such that $S_i \cap \overline{V(S_1, \varepsilon)}^c \neq \emptyset$ with $i \in \{2,3\}$. Then by Claim 2 $A \cap R^c \neq \emptyset$. Let $C$ be a connected component of $A \cap R^c$. So $f^j(C) \subset R^c$ for every $i \in \mathbb{Z}$. Consider the minimum of $i_0$ such that $f^{-i_0}(C) \subset B(q,\varepsilon)$. So there is $i \in \{1,2,...,i_0-1\}$ such that $C$ is not contained in $V(x_{-i}, \varepsilon)$. Thus $d_H(f^{-i}(A), x_{-i})> \varepsilon$ for some $i \in \{1,2,...,i_0-1\}$. Therefore, $C(f)$ does not have the shadowing property.

 
\end{proof}



\section{Topological entropy of the induced Morse-Smale diffeomorphisms}\label{Section-5}

This section is divided into topological entropy of the induced maps $2^f$ and $C(f)$ of Morse-Smale diffeomorphisms and sufficient conditions to obtain infinite topological entropy on the hyperspace. 

\subsection{Topological entropy of the induced maps $2^f$ and $C(f)$ of Morse-Smale diffeomorphisms}

In \cite{LR}, the authors obtain sufficient conditions to the topological entropy of the hyperspace map $2^f$ to be infinite. It can quickly be checked that these conditions hold for Morse-Smale diffeomorphisms on the circle. In the same way, we can prove that this property remains valid when the dimension of the manifold is greater than one as follows.

\begin{proposition}
Let $M$ be a $n$-dimensional compact and connected manifold without boundary. If $f: M \rightarrow M$ is a Morse-Smale diffeomorphism, then the topological entropy of its induced map $2^f$ is infinite.
\end{proposition}

\begin{proof}
Without loss of generality, we can assume that all periodic points of $f$ are fixed (otherwise we pass to some iteration of $f$). Consider one attractor fixed point $p$ and one repeller fixed point $q$ such that $W^s(p) \cap W^u(q)$ is a nonempty open set. In general, $W^s(p) \cap W^u(q)$ is not a connected set. So, let $W$ one open connected component of $W^s(p) \cap W^u(q) \neq \emptyset$ and let $\{y_k\}_{k=0}^{\infty}$ a countable subset of $W$. Fix an $r \in \mathbb{N}$. Then $L_r= \cup_{k=0}^{r-1} \overline {\emph{O}(y_k)} = \{f^n(y_k): 0 \leq k \leq r-1, n \in \mathbb{Z} \}  \cup \{p,q\}.$
Define a semiconjugacy $\phi$ between $2^{L_r}$ and $(\{0,1\}^r)^{\mathbb{Z}}$ as follows 
		
$\phi(A)= ((w_{0,n},w_{1,n},...,w_{r-1,n}))_{n \in \mathbb{Z}},$ 
$
\textrm{   where   } w_{k,n}=\left\{\begin{array}{rc}
1,&\mbox{if}\quad f^n(y_k) \in A ;\\
0,&\mbox{otherwise};\quad
\end{array}\right. 
$ \\
This map is continuous and the pre-image of every sequence in $(\{0,1\}^r)^{\mathbb{Z}}$ has at most four elements. Therefore, satisfies the equality $\phi \circ 2^f = \sigma \circ \phi$ where $\sigma$ is the shift map. Besides that $\sigma $ is conjugated to the two-sided on $2^r$ symbols, so that $h(2^f) \geq h(2^f|_{2^{L_r}})= \log (2^r)$ for any $r$. Thus $h(2^f)= \infty$. 
\end{proof}

The topological entropy of the induced map $C(f)$ has two possible values:\\


\textbf{Theorem B.} Let $f: M^n \rightarrow M^n$ be a Morse-Smale diffeomorphism, then the topological entropy of its induced continuum map $C(f)$ is zero or infinite.\\

Indeed, the topological entropy of the induced continuum map $C(f)$ depends on the dimension of the manifold. For this reason, to prove Theorem B, we will consider two cases: if the manifold is the circle $S^1$ and if $M^n$ is a manifold with dimension greater than two.\\

The first case, is a direct consequence of Lemma \ref{Theorem-nowandering-finite}, see Section \ref{Section-3}.   

\begin{theorem}
\label{entropy-zero-continuum}
Let $f: S^1 \rightarrow S^1$ be a Morse-Smale diffeomorphism, then the topological entropy of $C(f)$ is zero.
\end{theorem}

\begin{proof}
By Lemma \ref{Theorem-nowandering-finite}, the non-wandering set consists only of a finite number of periodic points and consequently the topological entropy is zero.
\end{proof}

\begin{remark}
Note that if $f: X \rightarrow X$ is a continuous map on a compact metric space, then for all $n \in \mathbb{N}$ the induced map $C_n(f)$ is a subsystem of $2^f$ which we denoted by $C_n(f) \leq 2^f$. So  
$$
C(f) \leq C_2(f) \leq C_3(f) \leq ... \leq 2^f.
$$
Therefore, if $f$ is a Morse-Smale diffeomorphisms on the circle we have that 
$$
h(C(f))=0 \leq h(C_2(f)) \leq h(C_3(f)) \leq ... \leq h(2^f)=\infty.
$$

Nevertheless, following the proof of Theorem \ref{Theorem-nowandering-finite} we can show $h(C_n(f))=0$ for all $n \in \mathbb{N}$. Thus, we have the following question.

 
\begin{question}
There exists a dynamical system $(X,f)$ such that 
$$
h(C(f)) < h(C_n(f)) < h(2^f)
$$ 
for some $n \geq 2$?
\end{question}

\end{remark}	

To the second case, we prove that if the dimension of $M$ is greater than or equal to $2$ then the topological entropy is infinite. For this, we follow the techniques used in Example $7$ in \cite{Abouda-Naghmouchi} by Abouda and Naghmouchi. In this article, the authors use the definition of topological entropy with separated sets, see Section \ref{section-preliminares}.
\begin{theorem} 
\label{Theorem-C(f)-infinite}
Let $M$ be a $n$-dimensional compact and connected manifold without boundary with $n \geq 2$. If $f:M \rightarrow M$ is a Morse-Smale diffemorphism, then the topological entropy of its induced map $C(f)$ is infinite.
\end{theorem}

\begin{proof}
Without loss of generality, we can assume that $f$ has only fixed points (otherwise we pass to some iteration of $f$). It is known there are an attractor $p$ and a repeller $q$ such that $W^s(p) \cap W^u(q) \neq \emptyset$.  Let $B^s$ be an open ball with center at $p$ such that $f(\partial B^s) \subset B^s$ where $\partial B^s= \overline{B}^s \setminus B^s$ is the boundary of $B^s$. Let $ Q^s=B^s \setminus f(\overline{B}^s)$ be the interior of a fundamental domain for the stable manifold of $p$. Similarly, we consider $B^u$ an open ball with center at $q$ such that $f^{-1}(\partial B^u) \subset B^u$  where $\partial B^u = \overline{B}^u \setminus B^u$ is the boundary of $B^u$ and $Q^u= B^u \setminus f(\overline{B}^u)$ is the interior of a fundamental domain for the unstable manifold of $q$. The Hartman-Grobman Theorem guarantees the existence of the sets $B^s$ and $B^u$, see \cite{PM} pg. 81. Since that $W^u(q)$ is a connected set, we have that there exists a point $x \in \partial B^s \cap W^u(q)$. Consider the point $x$, the point $f(x)$ and a connected component $N$ of $Q^s \cap W^u(q)$ with $x$ in its boundary. Notice that $N$ is an open set of $M$, see \cite{Kuratowski}  pg. 230. 
Set $y \in N$, an arc $\gamma$ with end points $x$ and $f(x)$ such that $y \in \gamma \setminus \{x,f(x)\} \subset Q^s$.  The arc $\gamma$ exits since $N$ is a region (a connected open set) and $n\geq 2$, see \cite{Kuratowski} pg. 230.  
\\
Let $m_0= \min \{n: f^{-n}(y) \in B^u\}$ and $\varepsilon> 0$ small enough such that:
\begin{itemize}
\item [i.] $B(y, \varepsilon) \subset N$
\item [ii.] $B(f^{-m_0}, \varepsilon) \subset B^u$,
\item [iii.] and $B(f^i(y), \varepsilon) \cap B(f^j(y), \varepsilon)= \emptyset$ for $i \neq j$ and $i,j \in \{0,-1,-2,...,-m_0\}$. 
\end{itemize}
 By the continuity of $f$, we can consider an arc $\beta \subset B(y, \varepsilon)$ such that $y$ is an end point of $\beta$, $\gamma \cap \beta=\{y\}$ and $f^j(\beta) \subset B(f^j(y), \varepsilon)$ for $j \in \{0,-1,...,-m_0\}$.
 
\begin{lemma}
For all $n \neq m \in \mathbb{Z}$, $f^n(\beta) \cap f^m(\beta)= \emptyset$.
\end{lemma}

\begin{proof}
Suppose that there exist $n_0, m_0 \in \mathbb{Z}$ such that $f^{n_0}(\beta) \cap f^{m_0}(\beta) \neq \emptyset$. Then $\beta \cap f^{m_0-n_0}(\beta) \neq \emptyset$. Thus, we have two possibilities. The first, if $m_0-n_0>0$ then $f^{m_0-n_0}(\beta) \subset f(\overline{B}^s)$ that is a contradiction. The second, if $m_0-n_0<0$ then $f^{m_0-n_0}(\beta) \subset B(f^{-1}(y), \varepsilon) \cup...\cup B(f^{-m_0+1}(y), \varepsilon) \cup B^u$ that is a contradiction.	
\end{proof}

\begin{lemma}
\label{Lema-continuum-MS}
There is $\varepsilon_1>0$ small enough such that  $V(\beta, \varepsilon_1) \cap f^n(\gamma)= \emptyset$ for all $n \in \mathbb{Z} \setminus \{0\}$, with $V(\beta, \varepsilon_1):=\{x \in M: d(x, \beta)<\varepsilon_1\}$.
\end{lemma}
	
\begin{proof}
Suppose that for every $k \in \mathbb{N}$ there is $n_k \in \mathbb{Z} \setminus \{0\}$ such that $f^{n_k}(\gamma) \cap V(\beta, \frac{1}{k}) \neq \emptyset$. Then we have that there exists a subsequence of $\{f^{n_k}(x_k)\}$, where $x_k \in \gamma$, such that 
$$
\lim _{k \rightarrow \infty} f^{n_k}(x_k) \in \beta \hspace{0.2cm} \textrm{or} \hspace{0.2cm} \lim _{k \rightarrow -\infty} f^{n_k}(x_k) \in \beta. 
$$
In the first case, we have there is $n_k$ big enough such that $f^{n_k}(x_k) \in  B(y, \varepsilon) \subset N$. Then $x_k=f^{-n_k}(f^{n_k}(x_k)) \in W^u(q) \setminus \overline{B}^s$ that is a contradiction. In the second case, we have there is $n_k$ big enough such that $f^{-n_k}(x_k) \in B(y, \varepsilon) \subset N$. Therefore, the point $x_k= f^{n_k}(f^{-n_k}(x_k)) \in f(B^s)$ and that is a contradiction. Thus we obtain the desired result.
\end{proof}

\begin{lemma}
If $k \in \mathbb{N}$, there is $\delta>0$ such that $s(n, C(f)^{-1},\delta) \geq k^n$ for all $n \in \mathbb{N}$.
\end{lemma}
	
\begin{proof}
Let $y$ and $z$ be the end points of the arc $\beta$. Fix $\bar{y} \in \beta$ close to the point $y$ and consider $\delta_0$ small enough such that $V([\bar{y},z],\delta_0)=\{x \in M: d(x,[\bar{y},z])< \delta_0\}$ does not contain points of $\gamma$. Since $[\bar{y},z]$ is an arc, it is homeomorphic to interval $[0,1]$. Let $H: [\bar{y},z] \rightarrow [0,1]$ be a homeomorphism such that $H(\bar{y})=0$ and $H(z)=1$. Let $\varepsilon_1>0$ be given by Lemma ~\ref{Lema-continuum-MS} and let $k \in \mathbb{N}$. There exists $\delta< \min\{\frac{1}{k},\varepsilon_1,\delta_0\}$ such that 
$$
d(x,y)< \delta \hspace{0.2cm} \textrm{implies} \hspace{0.2cm} |H(x)-H(y)|< \frac{1}{k}.
$$
Let $a_i=H^{-1}(i/k)$ for $i=1,...,k-1$. Let $n \in \mathbb{N}$, let $\sigma= (\sigma_0, \sigma_1...,\sigma_{n-1}) \in \{1,2,...,k\}^n$ and let $C_{\sigma}$ be the subtree of $\Lambda=\overline{\bigcup_{j \in \mathbb{Z}} f^{j}(\gamma \cup \beta)}$ defined as follow:
$$
C_{\sigma}=\bigcup_{j=0}^{n-1} f^j([a_{\sigma_j},y]) \cup  \bigcup_{j=0}^{n-2}f^j(\gamma).
$$
If $\sigma \neq \sigma^{'} \in \{1,2,...,k\}^n$, there is $j_0 \in \{0,1,...,n-1\}$ such that $\sigma_{j_0} \neq \sigma^{'}_{j_0}$. Without loss of generality, we can assume $\sigma_{j_0} <  \sigma^{'}_{j_0}$. Then
$$
\big| \frac{\sigma^{'}_{j_0}}{k} - t \big| >\frac{1}{k} \hspace{0.2cm} \textrm{for all} \hspace{0.2cm} t \in \big[0,\frac{\sigma_{j_0}}{k} \big].
$$
By the continuity of $H$ we have that
$$
d(a_{\sigma^{'}_{j_0}},H^{-1}(t)) > \delta \hspace{0.2cm} \textrm{for all} \hspace{0.2cm} t \in \big[0,\frac{\sigma_{j_0}}{k} \big].
$$
If it is necessary, we consider a smaller $\delta>0$ such that
$$
a_{\sigma^{'}_{j_0}} \in f^{-j_0}(C_{\sigma^{'}}) \hspace{0.2cm} \textrm{and} \hspace{0.2cm} a_{\sigma^{'}_{j_0}} \notin V(f^{-j_0}(C_{\sigma}),\delta).
$$  
Therefore $d_H(f^{-j_0}(C_{\sigma}),f^{-j_0}(C_{\sigma^{'}})) \geq \delta$. Thus, the collection of subtrees
$$
\{C_{\sigma}: \sigma \in \{1,...,k\}^n\} \subset C(\Lambda)$$
is $(n,C(f)^{-1},\delta)$-separated set  and $s(n,C(f)^{-1},\delta) \geq k^n$.	
\end{proof}
	
As a consequence, $h(C(f^{-1})) \geq \ln (k)$ for all $k \in \mathbb{N}$. By definition of topological entropy, we have $h(C(f))=h(C(f)^{-1})= \infty$.
\end{proof}

\subsection{Sufficient conditions to obtain infinite topological entropy in Hyperspace}\label{Section-5.2}

In this section we will present three sufficient conditions to obtain infinite topological entropy on the hyperspace $2^X$ and $C(X)$.

\subsubsection{Infinite topological entropy in $2^X$}

In \cite{DO}  the authors showed that the existence of homoclinic points in the base system generate positive entropy for the induced map $2^f$, see Theorem 13. With slight changes in the proof, it is possible to generalize Theorem 13 for a subshift of the two-sided shift $\Sigma_r$, with $r \geq 2$ large enough.  And thus it is possible to produce an arbitrary gap between the topological entropy of the base map and its induced map to the hyperspace as we will see in the following example.\\

\begin{ex}
Let $Q=\prod_{n \in \mathbb{Z}}J_n$ where $J_n=[0,1]$ for each $n \in \mathbb{Z}$ and let $D$ be a metric on $Q$ defined as follows:
$$D(\hat{t},\hat{s})=\sum_{n \in \mathbb{Z}} \frac{|t_n-s_n|}{2^{|n|}}$$
for $\hat{t}=(t_n)_{n \in \mathbb{Z}}$ and $\hat{s}=(s_n)_{n \in \mathbb{Z}}$. Let $\sigma: Q \rightarrow Q$ be the shift map. In \cite{AIM} it is shown some dynamical properties, in particular that $h(\sigma)=\infty$.  Now consider $S_Q$ the set of all bi-infinite sequences for which symbol $1/p$ occurs at most once, where $p \in \{1,2,...\}$. For $n \in \mathbb{Z}$ we denote by $b_n^p$ the point in $S_Q$ whose $n$th coordinate is the symbol $1/p$, and by $a$, we denote the sequence $(...,0,0,0,...) \in S_Q$. It is easy to see that $S_Q$ is a sub-shift of the full two-sided shift $Q$. We will write shortly $\sigma_{S_Q}$ for $\sigma|_{S_Q}$.   
\end{ex}
	
\begin{proposition}
\label{Infinite-entropy}
The subshift $(S_Q,\sigma_{S_Q})$ defined above: $h(\sigma_{S_Q})=0$ but $h(2^{\sigma_{S_Q}})= \infty$.
\end{proposition}  
	
\begin{proof} 
Since that $\Omega(\sigma_{S_Q})=\{a\}$, $\sigma_{S_Q}$ has the topological entropy equal zero. We will show that for any $r \geq 1$, the induced map $2^{\sigma_{S_Q}}$ is semi-conjugated to the full shift $\{\{0,1\}^r\}^{\mathbb{Z}}$. For the proof, we define the map $\phi: 2^{S_Q} \rightarrow \{\{0,1\}^r\}^{\mathbb{Z}}$ by
		
\vspace{0.2cm}
		
$\phi(A)= \{(y_n^1,y_n^2,...,y_n^r)\}_{n \in \mathbb{Z}},$ 
$
\textrm{   where   } y_{n}^i=\left\{\begin{array}{rc}
1,&\mbox{if}\quad p\leq r\quad \mbox{and}\quad  b^p_n \in A ;\\
0,&\mbox{otherwise};\quad
\end{array}\right. 
$ 
\vspace{0.2cm}
		
for any $A \in 2^{S_Q}$. It is easy to check that, for any $r \geq 1$, the map $\Phi$ is a semiconjugacy from $2^{\sigma_{S_Q}}$ to $\sigma$. Hence, $h(2^{\sigma_{S_Q}}) = \infty$.	
\end{proof}

Following the ideas of Kwietniak and Oprocha in \cite{DO} we show that Theorem \ref{finite-map}. As corollary we obtain Theorem $5.7$ in \cite{Paloma-Mendez}. This result was given before by Lampart-Raith in \cite{LR}.\\ 
	
\begin{theorem}
\label{finite-map}
\textit{Let $f: X \rightarrow X$ be a surjective map and let $X$ be a compact metric space. If there exists an infinite countable set $A= \{a_1,a_2,...\} \subset X$ such that}
\begin{itemize}
\item [i.]  \textit{$ L= \bigcup\limits_{i \geq 1} \alpha(\{x^{i}_{-n}\}_{n \in \mathbb{Z}_+},f) \cup \omega(a_i,f)$ and $M= \bigcup\limits_{i \geq 1} \{x^i_{-n}\}_{n \in \mathbb{Z}_+}$ are disjoint},
		
\item [ii.] \textit{For every pair $i \neq j$, $i \geq 1$, $j \geq 1$,
$$
Orb(\{x^i_{-n}\}_{n \in \mathbb{Z}_+},f) \cap Orb(\{x^j_{-n}\}_{n \in \mathbb{Z}_+},f) = \emptyset,
$$}
\end{itemize}
\textit{then $h(2^f)= \infty$}.
\end{theorem}
	
\begin{proof}
We will construct a semiconjugacy from some subsystem of $(X,f)$ to the one-side version of sub-shift, see Theorem 13 in \cite{DO}. First, fix an $r\in \mathbb{N}$. For each $i \in \{1,...,r\}$ extend the given negative orbit through $a_i$ to the full $Orb(\{x^{i}_n\}_{n \in \mathbb{Z}_+},f)$. Then define the closed $f$-invariant set as 
$$
\Lambda= \bigcup_{i=0}^{r-1} \overline{Orb(\{x^{i}_n\}_{n \in \mathbb{Z}_+},f)}.
$$
With the notation introduced in Proposition \ref{Infinite-entropy},  we define a map $\phi: \Lambda \rightarrow S_Q$ by
\vspace{0.2cm}
		
$
\phi(x)=\left\{\begin{array}{rc}
b^i_n,&\mbox{if}\quad x=x^i_{-n} \textrm{  for some  }  n \geq 0 \textrm{   and   }  1 \leq i \leq r ;\\
a,&\mbox{otherwise.}\quad
\end{array}\right. 
$ 
\vspace{0.2cm}
		
Clearly $\phi$ is the desired semiconjugacy. This implies that $2^{\Phi}$ is a semiconjugacy from $2^f$ to $2^{\sigma_{S_Q}}$. Therefore, $h(2^f) \geq \log r$ for all $r \in \mathbb{N}$.
\end{proof}

\subsubsection{Case 1: Infinite topological entropy on $C(X)$}
In this section we present a first mechanisms to ensure infinite topological entropy on $C(X)$.\\
	
	
In \cite{AIM} the authors present two examples of dynamical systems defined on dendrites such that the continuum induced map has infinite topological entropy. We are interested in one of these examples and we will present it here.

\textbf{Example 1}	
First consider, in $\mathbb{R}^2$, the points $p=(-1,0)$, $q=(1,0)$ and the sequence $((a_n,0))_{n \in \mathbb{Z}}$ such that $a_0=0$ and, for each $n \in \mathbb{N}$, $a_n= 1-\frac{1}{n+1}$ and $a_{-n}=-a_n$. Note that $a_n < a_{n+1}$ for each $n \in \mathbb{Z}$. Moreover:
$$\lim_{n \rightarrow \infty} a_n=1 \\\ \textrm{   and} \\\ \lim_{n \rightarrow - \infty} a_n= -1.$$
	
Now, given $n \in \mathbb{Z}$, let $L_n$ be the straight line segment $L_n=\{a_n\} \times [0, \frac{1}{|n|+1}]$. We denote the segment $[-1,1] \times \{0\}$ by $pq$. Let 
$$X=pq \cup \left (\bigcup_{n \in \mathbb{Z}}L_n \right).$$

	
Note that $X$ is a dendrite with free arcs. Let $F: X \rightarrow X$ be a homeomorphism with the following properties:
	
\begin{itemize}
\item [i.] $F|_{pq}$ is a homeomorphism from $pq$ onto itself such that $F(p)=p$, $F(q)=q$ and $F((a_n,0))=(a_{n+1},0)$ for each $n \in \mathbb{Z}$ (so the image under $F$ of the arc from $(a_n,0)$ to $(a_{n+1},0)$ is the arc from $(a_{n+1},0)$ to $(a_{n+2},0$);
\item [ii.] for each $n \in \mathbb{Z}$, $F|_{L_n}: L_n \rightarrow L_{n+1}$ is a linear homeomorphism.
\end{itemize} 
	

Note that $Fix(F)=\{p,q\}$. Moreover:
$\omega(p,F)=\{p\}$ and $\omega(x,F)=\{q\}$, for each $x \in X-\{p\}$. Thus $F$ is not transitive. Since $X$ is a dendrite and $F$ is a homeomorphism, by \cite{AEO} the topological entropy of $F$ is zero. Now, since $F$ is a homeomorphism, the induced map $C(F): C(X) \rightarrow C(X)$ is also a homeomorphism, see \cite{IN} . By either Theorem $4.5$ or Theorem $6.2$ in \cite{AIM}, $C(F)$ is not transitive. Note that, $\omega(A,C(F))= \{\{q\}\}$, for each $A \in C(X)$ such that $p \notin A$.
	
The authors in \cite{AIM} constructed a closed subset $\Lambda$ of $C(X)$ such that $C(F)_{|_{\Lambda}}$ is topologically conjugate to \textit{the shift map} defined on the Hilbert cube. For this reason, $C(F)|_{\Lambda}$ has the following properties
	
\begin{itemize}
\item [i.] The homeomorphism $C(F)|_{\Lambda} : \Lambda \rightarrow \Lambda$ is \textit{Devaney chaotic}, i.e., is transitive and periodically dense.
\item [ii.] The topological entropy of $C(F)|_{\Lambda}$ is infinite, so $C(F)|_{\Lambda}$ and $C(F)$ are topologically chaotic. 
\item [iii.] $C(F)|_{\Lambda}$ has uncountable periodic points of each period.
\end{itemize} 
	
We observed that there are hyperbolic dynamical systems with this dendrite. Therefore, we will introduce the definition of special dendrite for homeomorphisms.
	
\begin{definition}
\label{def-sp-den}
Let $f:X \rightarrow X$ be a homeomorphism on a compact metric space $X$. We say that a closed subset $\Lambda \subset X$ is a \textit{Special Dendrite} if there is $k \in \mathbb{N}$ such that $\Lambda$ is $f^k$-invariant and $f^k|_{\Lambda}$ is conjugated to $F$ (homeomorphism defined above). In this case, we say that $f$  \textit{admit a Special Dendrite}.
\end{definition}
	
Let us present some examples of dynamical systems that admit a special dendrite, and note that the base space is a $n$-dimensional manifold with $n\geq 2$.\\

\begin{ex}
The set $MS_2(M^n)$ of Morse-Smale diffeomorphisms on a compact manifold without boundary and with only two fixed points is non-empty, when the manifold is an $n$-dimensional sphere. In this case, the dynamics is very simple: all non-fixed points move from the source to the sink. On a sphere of a given dimension, any two such diffeomorphisms are topologically conjugate. We will show that every Morse-Smale diffeomorphism with only two hyperbolic fixed points admits a special dendrite.
	

	
\begin{proposition}
\label{MS-2}
Let $M$ be a compact, connected and orientable $n$-dimensional manifold without boundary and let $n \geq 2$. If $f \in MS_2(M^n)$ then $f$ admits a special dendrite.
\end{proposition}
	
\begin{proof}
Let $f \in MS_2(M^n)$. Since that, all non-fixed points move from the source $\bar{q}$ to the sink $\bar{p}$, we have $W^s(\bar{p})\bigcap W^u(\bar{q})= M-\{\bar{p},\bar{q}\}$ is an open and connected $n$-dimensional set.  Let $\varepsilon>0$ be sufficiently small such that
\begin{itemize}
\item $B(\bar{p},\varepsilon) \bigcap B(\bar{q},\varepsilon)= \emptyset$.
\item $f(\overline{B(\bar{p},\varepsilon)}) \subset B(\bar{p},\varepsilon)$ and
\item $f^{-1}(\overline{B(\bar{q},\varepsilon)}) \subset B(\bar{q},\varepsilon)$.
\end{itemize}   
Let $Q^s:= \overline{B(\bar{p}, \varepsilon)} - f(\overline{B(\bar{p}, \varepsilon)})$ be a fundamental domain for the stable manifold of $\bar{p}$.

Consider $x \in \partial B(\bar{p}, \varepsilon)$, $f(x) \in \partial f(B(\bar{p}, \varepsilon))$ and compact arcs $\gamma$ and $\beta$ such that:
\begin{itemize}
\item the end points of $\gamma$ are the points $x$ and $f(x)$  and 
$$\gamma-\{x, f(x)\} \subset int(Q^s)$$ 

\item one of the end points of $\beta$ is $x$ and 
$$\beta \subset Q^s \textrm{ and } \gamma \cap \beta=\{x\}$$
\end{itemize}

%
%

\textbf{Claim 1:}
For all $n \neq m \in \mathbb{Z}$, $f^n(\gamma) \cap f^m(\gamma)$ is an empty set or one point of the orbit of $x$.

\begin{proof}
Suppose there are $n \neq m \in \mathbb{Z}$ and $ f^n(\gamma) \cap f^m(\gamma)$ is not an empty set. Then there is $ y \in f^n(\gamma) \cap f^m(\gamma)$. We can assume $n>m$. Then $ f^{-m}(y) \in f^{n-m}(\gamma) \cap \gamma.$ If $n-m>1$ then $f^{n-m}(\gamma) \cap \gamma = \emptyset$ since $\gamma \subset Q^s$, therefore $n-m=1$ and $ f^{-m}(y) \in f(\gamma) \cap \gamma=\{f(x)\}$. Thus $y=f^{m+1}(x)$ .
\end{proof}

\textbf{Claim 2:}
There is $t_0 \in \mathbb{N}$ such that 
$$f^{-t_0}(\gamma \cup \beta ) \subset B(\bar{q},\varepsilon)\\\ \textrm{and}\\\ f^{t_0}(\gamma \cup \beta) \subset B(\bar{p},\varepsilon).$$

\begin{proof}
It is sufficient to prove that there is $t_0  \in \mathbb{N}$ such that 
$$f^{-t_0}(\gamma \cup \beta) \subset B(\bar{q},\varepsilon) .$$
Suppose that, for all $k \geq 1$ there is $y_k$ in $\gamma \cup \beta$ such that $f^{-k}(y_k)$ is not contained in $B(\bar{q},\varepsilon)$. Then we obtain two sequence $\{y_k\}_{k \geq 1}$ in $\gamma \cup \beta$ and $\{f^{-k}(y_k)\}_{k \geq 1}$ in $B(q,\varepsilon)^c$. Taking a subsequence, if it is necessary, we have $\lim\limits_{k \rightarrow \infty} y_k= \bar{x} \in \gamma \cup \beta$. For $\bar{x}$, there exists $m_0 \in \mathbb{N}$ such that $f^{-m_0}(\bar{x}) \in B(\bar{q},\varepsilon)$. Let $\varepsilon_1 >0$ be such that $B(f^{-m_0}(\bar{x}), \varepsilon_1) \subset B(\bar{q},\varepsilon)$, so since $f$ is continuous, there exists, $\delta>0$ such that 
$$
f^{-m_0}(B(\bar{x},\delta)) \subset B(f^{-m_0}(\bar{x}), \varepsilon_1).
$$
Now, let $n_0 > m_0$ be such that  $y_{n_0} \in B(\bar{x},\delta)$  and $f^{-m_0}(y_{n_0}) \in B(f^{-m_0}(\bar{x}),\varepsilon_1)$. But $ B(f^{-m_0}(\bar{x}),\varepsilon_1) \subset B(\bar{q},\varepsilon)$. Thus, we have $f^{-(m_0+t)}(y_{n_0}) \in B(\bar{q},\varepsilon)$ for all $t \geq 1$, which is a contradiction. 
\end{proof}


\textbf{Claim 3:}
For all $n \neq m \in \mathbb{Z}$, $f^n(\beta) \cap f^m(\beta)= \emptyset$.

\begin{proof}
As in the proof of Claim 2, since $Q^s$ is a fundamental domain we have that if $n\neq 0$ then $f^n(Q^s)\cap Q^s=\emptyset$. The claim follows since $\beta\in Q^s$.
\end{proof}

Now we will construct a special dendrite as following:
		
Let $\Lambda= \overline{\bigcup\limits_{n \in \mathbb{Z}}f^n(\gamma \cup \beta)}=  \bigcup\limits_{n \in \mathbb{Z}}f^n(\gamma \cup \beta) \bigcup \{\bar{p},\bar{q}\}$ be a compact and $f$-invariant set.

\textbf{Claim 4:}
$f|_{{\Lambda}}$ is conjugate to $F$ ( see Example 1). 

\begin{proof}
Let $A_0=\gamma\cup\beta-\{f(x)\}$ and $B_0=L_0\cup (-1/2,0]\times \{0\}$. Let $y$ be the endpoint of $\beta-\{x\}$.

We fix a homeomorphism $G:A_0\to B_0$ such that, $G(x)=(0,0)$, $G(\gamma-{f(x)})=(-1/2,0]\times \{0\}$, $G(\beta)=L_0$ and $G(y)=(0,1)$.

If $n\in\mathbb{Z}$ we set $A_n=f^n(A_0)$ and $B_n=F^n(B_0)$. By the claims, if $n\neq m$ then $A_n\cap A_m=\emptyset$. Also, by Example 1  if $n\neq m$ then $B_n\cap B_m=\emptyset$.

Thus, we can define a homeomorphism $H:A_n\to B_n$ as $H=F^n\circ G\circ f^{-n}$. Finally, we define $H(\bar{p})=p$ and $H(\bar{q})=q$. This define a bijection $H: \Lambda \rightarrow pq \cup \left (\bigcup_{n \in \mathbb{Z}}L_n \right)$. Moreover, $H$ is continuous in $\Lambda-(\cup_{n\in \mathbb{Z}}\{f^n(x)\}\cup\{\bar{p},\bar{q}\})$.
\end{proof}

\textbf{Claim 5:}
$H$ is continuous at $f^n(x)$ for every $n\in\mathbb{Z}$.
\begin{proof}
We first prove the claim for $n=1$, the same argument holds in the general case. 

If $z_n\in\gamma-\{f(x)\}$ and $z_n\to f(x)$, then $H(z_n)=G(z_n)\to (-1/2,0)$. If $z_n\in A_1$ and $z_n\to f(x)$ then by definition \
$$H(z_n)=F\circ G\circ f^{-1}(z_n)\to F( G(f^{-1}(f(x)))=F(G(x))=F( (0,0))=(-1/2,0).$$
Thus $H$ is continuous in $f(x)$.
\end{proof}

\textbf{Claim 6:} $H$ is continuous at $\bar{p}$ and $\bar{q}$.

\begin{proof}
Let $z_n\to \bar{p}$ and $\eps>0$. Let $k_0$ such that $\bigcup_{k\geq k_0}B_k\subset B(p,\eps)$. Since $\bar{p}$ is a sink, there exists $n_0$ such that if $n\geq n_0$ then $z_n\in \bigcup_{k\geq k_0}A_k\cup\{\bar{p}\}$. Thus $H(z_n)\in  B(p,\eps)$ for $n\geq n_0$. Then $H$ is continuous in $\bar{p}$. The continuity at $\bar{q}$ is analogous. 
\end{proof}




		



Thus, $H$ is a homeomorphism since is a continuous bijection on a compact set. Moreover, by construction $H \circ f = F \circ H$. Therefore, $f$ admits a special dendrite.
\end{proof}
\end{ex}

\begin{ex} 
Let $MS_3^*(M^n)$ be the subclass of $MS_3(M^n)$ of diffeomorphisms with only three periodic points. The first question about this subclass of diffeomorphisms is if this set is non-empty. It is well known that there are no Morse-Smale diffeomorphisms on $2$-manifolds with exactly three periodic points. In \cite{G-Z-M}, the authors showed that, there are no Morse-Smale diffeomorphisms on $3$-manifolds whose set of non-wandering points consists of exactly three periodic points. In \cite{Z-M}, the existence of closed $n$-manifolds with $n \geq 4$ admitting Morse functions with precisely three critical points was proved, and such manifolds were studied. Thus, in the case $n \geq 4$, there exists Morse-Smale diffeomorphisms with precisely three periodic points. In \cite{Z-M}, the authors showed that any such diffeomorphism has precisely one saddle, one sink and one source. Also, they showed that if $n$ is an even number and $n \geq 4$, then the unstable and the stable separatrix are $n/2$-dimensional spheres. Thus we obtain the following result.
	
\begin{proposition}
Let $M^n$ be a compact and connected, orientable $n$-dimensional manifold with $n \geq 4$. Let $f:M^n \to M^n$ be a Morse-Smale diffeomorphism with only three periodic points, an attractor $p$, a repeller $q$ and a saddle $\sigma$. Then $f$ admits a special dendrite. 
\end{proposition} 
	
\begin{proof}
Without loss of generality, we can assume that all periodic points of the diffeomorphism $f$ are fixed (otherwise we pass to some iteration of $f$). By the Main Theorem in \cite{Z-M}, we have that 
$$
W^u(\sigma) \cup \{p\}=S_p \hspace{0.3cm} \textrm{and} \hspace{0.3cm} W^s(\sigma) \cup \{q\}=S_q
$$
are $n/2$-dimensional spheres. In this case the dynamic is very simple: all non-fixed points move from the source to the saddle or from the saddle to the sink. Therefore, by Proposition ~\ref{MS-2}, we have that $f$ admits a special dendrite.
\end{proof}

\begin{proposition}
\label{component-dendr-sp}
Let $M^n$ be a compact and connected, orientable $n$-dimensional manifold with $n \geq 4$. Let $f:M^n \to M^n$ be a Morse-Smale diffeomorphism with an attractor point $p$ and a repeller point $q$ such that $W^s(p) \cap W^u(q) \neq \emptyset$ has a $f^k$-invariant connected component with dimension at least $2$. Then $f$ admits a special dendrite.    
\end{proposition}
	
\begin{proof}
Since that the set $W^s(p) \cap W^u(q) \neq \emptyset$ has dimension at least $2$ and it is $f^k-$ invariant, we can construct a special dendrite with the same arguments in Proposition ~\ref{MS-2}. 
\end{proof}
\end{ex}


\begin{ex} 	
Consider the torus $T^2 \subset \mathbb{R}^3$ and let $X = grad (t)$ where $t$ is the height function of points of $T^2$ above the horizontal plane. This	vector field has four singularities $p_1$, $p_2$, $p_3$, $p_4$ where $p_1$ is a sink, $p_2$ and $p_3$ are saddles and $p_4$ is a source. The stable manifold of $p_2$ intersects the unstable manifold of $p_3$ nontransversally. The diffeomorphism time-one map of $X$ admits a special dendrite, see definition in \cite{PM}.\\
If It is destroyed the intersection nontransversally with a small perturbation of the field $X$, then the resulting field $Y$ is a Morse-Smale field and therefore not be equivalent to $X$. The diffeomorphism time-one map of $Y$ admits a special dendrite.\\
\end{ex}

Homeomorphisms that contain an special dendrite in the base space generate infinite topological entropy in the hyperspace. 
	
\begin{theorem}
\label{criterio-special-chaotic-dendrite}
\textit{Let $f: X \rightarrow X$ be homeomorphism on a continuum metric space $X$. If $f$ admits a special dendrite, then $h(C(f))= \infty$ and therefore $h(2^f)=\infty$}.
\end{theorem}
	
\begin{proof}
Since that $f$ admits a special dendrite, there exist a closed subset $\Lambda \subset X$ and $k \in \mathbb{N}$ such that $\Lambda$ is $f^k$-invariant and $f^k|_{\Lambda}$ is conjugate to $F$. So the induced continuum systems $C(f^k|_{\Lambda})$ and $C(F)$ are conjugate too (see \cite{RC}, Theorem $4$). Therefore, we have that $h(f) \geq h(C(F))=\infty$. 
\end{proof}


\subsubsection{Case 2: Infinite topological entropy on $C(X)$}	
First we give an example of a dynamical systems with zero topological entropy and  such that  the continuum map has infinite topological entropy. Let $(Cone(S), \sigma_c)$ the Example 15, in \cite{DO}. The authors showed that $\sigma_c$ has zero topological entropy and the induced continuum map  $C(\sigma_c)$ has positive topological entropy. We will show the topological entropy of $C(\sigma_c)$ is infinite.\\
	
	
\begin{proposition}
Let $(Cone(S), \sigma_c)$ be the system in example above. Then the topological entropy of $C(\sigma_c)$ is infinite. 
\end{proposition}
	
\begin{proof}
Let $S_1= \{a, b^1_n\}_{n \in \mathbb{Z}}$. Given $n \in \mathbb{Z}$, let $L_n= \{b^1_n\} \times [0,1]$. Note that, $Cone(S_1)$ is a compact and connected $\sigma_c$-invariant set. We denote $[1]= S \times \{1\}$, the fixed point in $Cone(S_1)$. Now, consider the following set:
$$
\Lambda=\{A \in C(Cone(S_1)): [1] \in A\}
$$
Note that $\Lambda$ is a closed subset of $C(Cone(S))$. Since $\sigma_c([1])=[1]$, $\Lambda$ is strongly invariant under $C(\sigma_c)$. Consider $\sigma: Q \rightarrow Q$ \textit{the shift map} defined on a Hilbert Cube, $Q=\Pi_{n \in \mathbb{Z}} [0,1]$. We will show the induced map $C(\sigma_c)|_{\Lambda}$ is semi-conjugated  to \textit{the shift map} $\sigma$ defined on the Hilbert cube. For $A \in \Lambda$, let $\phi(A)= \hat{(t)}= (t_n)_{n \in \mathbb{Z}}$, where $t_n= 1- \min (\pi_2(L_n \cap A))$, for every $n \in \mathbb{Z}$. In this way we have a function $\phi: \Lambda \rightarrow Q$. Note that, $\phi$ is continuous and onto. Moreover $\phi \circ C(\sigma_c)|_{\Lambda} = \sigma^{-1} \circ \phi$ since $\sigma_c(b_n^i,a)=(b_{n+1}^i,a)$ for any $(b_n^i,a)$ in $Cone(S)$. Thus, the topological entropy of $C(\sigma_c)$ is greater than or equal to the topological entropy of $\sigma$. Since the topological entropy of $\sigma$ is infinite, by Theorem $7.6$ in \cite{AIM}, we conclude that the topological entropy of $C(\sigma_c)$ is infinite.
\end{proof}

	
Now, Let $f:M \rightarrow M$ be a homeomorphism on a continuum metric space $M$, with three fixed points $p$, $q$ and $\sigma$. Fix $r \in \mathbb{N}$ and suppose that there exist $a_0, a_1,..., a_{r-1} \in M$ with the following property: $\lim_{n\rightarrow \infty}f^n(a_i)= p$ and $\lim_{n\rightarrow -\infty}f^n(a_i)= q$ for every $i \in \{0,...,r-1\}$. Also, suppose there exist arcs $\gamma_i$ joining the points $a_i$ and $\sigma$ such that
$$
f^{n_1}(\gamma_{i_1}) \cap f^{n_2}(\gamma_{i_2}) = \{ \sigma \} \hspace{0.2cm} \textrm{for all} \hspace{0.2cm} (i_1,n_1) \neq (i_2,n_2). 
$$
	
In this section, we will denote that $f^n(\gamma_i)= \gamma_{i,n}$ and $\hat{\gamma}_{i,n}:=\gamma_{i,n} \backslash \{ \sigma\}.$ 
	
\begin{definition}
We say that the sequence $\{f^n(\gamma_i) \}^{i \in \{0,...,r-1\}}_{n \in \mathbb{Z}}$ \textit{ is Self-Accumulated} if there exist $x \in \bigcup_{n \in \mathbb{Z}}f^n( \gamma_0 \cup \gamma_1 \cup \ldots \cup \gamma_{r-1})$, a sequence of points $\{x_j\}_{j \in \mathbb{N}} \subset \gamma_0 \cup \ldots \cup \gamma_{r-1}$ and $\{k_j\}_{j \in \mathbb{N}} \subset \mathbb{Z}$, $|k_j| \rightarrow \infty$ when $j \rightarrow \infty$, such that $f^{k_j}(x_j) \rightarrow x$ when $j \rightarrow \infty$.
\end{definition}
	
\textbf{Remark.}
\label{OBS2}
If the sequence $\{f^n(\gamma_i)\}^{i \in \{0,...,r-1\}}_{n \in \mathbb{Z}}$ is not self-accumulated then, for any $(i_0,n_0) \in \{0,...,r-1\} \times \mathbb{Z}$ and each $x \in \hat{\gamma}_{i_0,n_0}$, there exists $\varepsilon_0 >0$ such that 
$$
B(x, \varepsilon_0) \cap \gamma_{i,n} = \emptyset \hspace{0.2cm} \textrm{for all} \hspace{0.2cm} (i,n) \neq (i_0,n_0). 
$$

From now, in this section, we will assume that $M$ is a continuum metric space and $f:M \rightarrow M$ is a homeomorphism with three fixed points $p$, $q$ and $\sigma$. Besides that, we suppose that there exist an infinite countable set $A=\{a_0,a_1,...\} \subset M$ with $\alpha(a_i)=\{q\}$ and $\omega(a_i)= \{p\}$ for all $i \geq 0$, and a sequence of arcs $\{\gamma_i\}_{i \geq 0}$, where $\gamma_i$ joins the points $a_i$ and $\sigma$ such that for every $r \in \mathbb{N}$ the sequence $\{f^n(\gamma_i) \}^{i \in \{0,...,r-1\}}_{n \in \mathbb{Z}}$ is not self-accumulated.  For all $i \in \{0,...,r-1\}$ we set 
$$
\Lambda_i= \{x \in M: \exists \\\ n_j \rightarrow \infty \textrm{  and} \\\ x_{i,j} \in \gamma_{i,n_j} \textrm{  such that} \\\ x_{i,j} \rightarrow x  \textrm{ when} \\\ j \rightarrow \infty \}
$$
and
$$
\Gamma_i= \{x \in M: \exists \\\ n_j \rightarrow -\infty \textrm{  and} \\\ x_{i,j} \in \gamma_{i,n_j} \textrm{  such that} \\\ x_{i,j} \rightarrow x \textrm{ when} \\\ j \rightarrow \infty \}
$$
Note that, these sets are compact subsets of $M$ by definition.
	
\begin{lemma}
\label{Limit-sets}
For all $i \in \{0,...,r-1\}$, we have that 
$$
\Lambda_i= \bigcup_{A_{\lambda} \in  \omega(\gamma_{i,0})} A_{\lambda} 
\: \: \: \: \: \: \: \: 
\textrm{and}
\: \: \: \: \: \: \: \: 
\Gamma_i= \bigcup_{B_{\lambda} \in  \alpha(\gamma_{i,0})} B_{\lambda}
$$
where $\omega(\gamma_{i,0})= \omega(\gamma_{i,0}, C(f))$ and $\alpha(\gamma_{i,0})= \alpha(\gamma_{i,0}, C(f))$. Therefore, $\Lambda_i$ and $\Gamma_i$ are compact and connected subsets of $M$ such that $q,\sigma \in \Lambda_i$ and $p, \sigma \in \Gamma_i$.
\end{lemma}
	
\begin{proof}
First, suppose that $x \in A_{\lambda}$ with $A_{\lambda} \in \omega(\gamma_{i,0},C(f))$. Then there is a subsequence $\gamma_{i,n_j}$ such that $d_H(\gamma_{i,n_j},A_{\lambda}) \rightarrow 0$. From the definition of Hausdorff metric, we have that there is a subsequence $y_{i,n_j} \in \gamma_{i,n_j}$ such that $d(x,y_{i,n_j}) \rightarrow 0$. Therefore, $ x \in \Lambda_i$. Now, suppose that $x \in \Lambda_i$, then there are $n_j \rightarrow \infty$ and $x_{i,j} \in \gamma_{i,n_j}$ such that $x_{i,j} \rightarrow x$ when $j \rightarrow \infty$. Since $C(M)$ is a compact metric space (see  \cite{IN}), there are $A_{\lambda_0} \in \omega(\gamma_{i,0},C(f))$ and a subsequence of $\{\gamma_{i,n_j}\}_{j \geq 1}$ such that  $d_H(\gamma_{i,n_j}, A_{\lambda_0}) \rightarrow 0$. We claim that $x \in A_{\lambda_0}$. Indeed, by definition of Hausdorff metric, we have that $d(x_j,A_{\lambda_0}) \rightarrow 0$ so by triangle inequality we have $d(x,A_{\lambda_0})=0$. Therefore $x \in A_{\lambda_0}$. We use the same argument to show the second equality. 
\end{proof}
	
Note that, for all $i \in \{0,...,r-1\}$ the set $\Lambda_i \cup \Gamma_i$ is a fixed point of $C(f)$ since that $\omega(\gamma_{i,0},C(f))$ is strongly invariant set of $C(f)$. We set
$$
\Delta_r =  \bigcup^{r-1}_{i=0} \Lambda_i \cup \Gamma_i
$$
and	
$$
L_r = \bigcup_{n \in \mathbb{Z}} \left( \bigcup_{i=o}^{r-1} \gamma_{i,n} \cup \Delta_r \right).
$$
	
Note that, both $\Delta_r$ and $L_r$ are a compact, connected and $f$-invariant subsets of $M$. Therefore, $C(L_r)$ is a subset of $C(M)$. We denote $\hat{\gamma}_{i,n}= \gamma_{i,n}\backslash\{ \sigma\}$ and we say that a set  $K\in C(L_r)$ is a \textit{Full Cone} if it satisfies the following properties:
\begin{itemize}
\item [i.]  If $K\cap \hat{\gamma}_{i,n}\neq \emptyset$ then $\gamma_{i,n} \subset K$,
\item [ii.] $\Delta_r \subset K$.
\end{itemize}
	
We denote by $H_r$ the set of all \textit{Full Cones} in $C(L_r)$. Note that, there are only two types of full cones in $H_r$. We say that $K \in H_r$ is a \textit{Finite Full Cone} if the set $\Sigma=\{(i,n) \in \{0,...,r-1\} \times \mathbb{Z}: K \cap \hat{\gamma}_{i,n} \neq \emptyset\}$ is finite or \textit{Infinite Full Cone} if $\Sigma$ is infinite. Therefore, if $K$ is a full cone we can write this as 
$$
K= \bigcup\limits_{(i,n) \in \Sigma} \gamma_{i,n} \cup \Delta_r.
$$

\begin{lemma}
\label{CI}
The subset $H_r$ is a closed subset of $C(L_r)$ and $C(f)$-invariant.
\end{lemma}

\begin{proof}
First, we will prove that $H_r$ is a closed subset of $C(L_r)$. Indeed, let $\{K_m\}_{m \in \mathbb{N}}$ be a sequence in $H_r$ such that $d_H(K_m,K_0) \rightarrow 0$ if $m \rightarrow \infty$. First, note that for all $m \in \mathbb{N}$, $\Delta_r \subset K_m$, therefore $\Delta_r \subset K_0$. Now, suppose that for some $i_0 \in \{0,...,r-1\}$ and some $n_0 \in \mathbb{Z}$, $\hat{\gamma}_{i,n_0} \cap K_0 \neq \emptyset$. We claim that $\hat{\gamma}_{i_0,n_0} \subset K_0$. Indeed, let $x \in \hat{\gamma}_{i_0,n_0} \cap K_0$. So there exist $\varepsilon_0>0$ such that $B(x,\varepsilon_0) \cap \gamma_{i,n}= \emptyset$ for any $i \neq i_0$ and $n \neq n_0$. Note that, $B(x, \varepsilon_0) \cap \Delta_r = \emptyset$. Then for all $\varepsilon < \varepsilon_0$, there exists $m_0 \in \mathbb{N}$ such that $d_H(K_m,K_0) < \varepsilon$ for each $m \geq m_0$. Thus, for $x$ there exists $y_m \in K_m$ with $d(x,y_m) < \varepsilon$, but this implies $y_m \in \gamma_{i_0,n_0}$ and, therefore, $\gamma_{i_0,n_0} \subset K_m$. This implies that for every $\varepsilon >0$, $\gamma_{i_0,n_0} \subset V(K_0, \varepsilon)$. So we conclude that $\gamma_{i_0,n_0} \subset K$, i.e, $K_0$ is a full cone. 
		
Now, we will prove that $H_r$ is a $C(f)$ invariant set. Let $K$ be a full cone, then 
$$
K= \bigcup_{(i,n) \in \Sigma}\gamma_{i,n} \cup \Delta_r .
$$
Thus, we have that 
$$
C(f)(K)= f(\bigcup_{(i,n) \in \Sigma} \gamma_{i,n} \cup \Delta_r)= \bigcup_{(i,n) \in \Sigma} f(\gamma_{i,n}) \cup \Delta_r = \bigcup_{(i,n) \in \Sigma} \gamma_{i,n+1} \cup \Delta_r.
$$
Therefore, $C(f)(K)$ is a full cone. Using the same argument above for $C(f^{-1})$, we obtain that $H_r$ is $C(f)$ invariant.
\end{proof}
	
By Lemma ~\ref{CI}, we have that $(C(f),H_r)$ is a subsystem of $(C(f),C(M))$. Now, we will show that $C(f^{-1})$ contain a full shift of $2^r$ symbols. 
	
\begin{lemma}
\label{Semiconj}
The induced map $C(f^{-1}): H_r \rightarrow H_r$ is conjugate to  the full shift $\sigma:(\{0,1\}^r)^{\mathbb{Z}} \to (\{0,1\}^r)^{\mathbb{Z}}$.	
\end{lemma}
	
\begin{proof}
Let $\phi: H_r \to (\{0,1\}^r)^{\mathbb{Z}}$ be a map defined by 
\vspace{0.18 cm}
		
$\phi(K)= ((w_{1,n},w_{2,n},...,w_{r,n}))_{n \in \mathbb{Z}} $ 
$
\textrm{   where   } w_{i,n}=\left\{\begin{array}{rc}
1,&\mbox{if}\quad \hat{\gamma}_{i,n} \cap K \neq \emptyset;\\
0, &\mbox{otherwise};\quad
\end{array}\right. 
$ 
		
for every $K \in H_r$.
		
It is clear that $\phi$ is onto and injective. We claim that $\phi$ is continuous. In fact, let  $\{K_j\}_{j \in \mathbb{N}}$ be a sequence of full cones such that $d_H(K_j,K) \rightarrow 0$ when $j \rightarrow \infty$. Since the sequence $\{ \gamma_{i,n} \}^{i \in \{0,...,r-1\}}_{n \in \mathbb{Z}}$ is not self-accumulated, for any $\varepsilon >0$, small enough, there exist $j_0, n_0 \in \mathbb{N}$ such that if $|n| \leq n_0$, $i \in \{0,...,r-1\}$ and $\hat{\gamma}_{i,n} \subset K$ then $\hat{\gamma}_{i,n} \subset K_j$ for every $j \geq j_0$. Therefore, $d(\phi(K_j), \phi(K)) \rightarrow 0$ when $j \rightarrow \infty$ and thus $\phi$ is continuous. The fact that $\sigma \circ \Phi= \Phi \circ C(f^{-1})|_{H_r}$ follows from:
$$
\hat{\gamma}_{i,n} \cap K \neq \emptyset \Leftrightarrow f^{-1}(\hat{\gamma}_{i,n}) \cap f^{-1}(K) \neq \emptyset \Leftrightarrow \hat{\gamma}_{i,n-1} \cap f^{-1}(K) \neq \emptyset.
$$
\end{proof}
	
We are ready to show the main result.\\
	
\begin{theorem}
\label{t-Alexander}
\textit{Let $M^n$ be a compact, connected $n$-dimensional space with $n \geq 2$ and let $f: M^n \rightarrow M^n$ be a homeomorphism with three fixed points $p,q$ and $\sigma$. If there exists an infinite countable set $A=\{a_0,a_1,a_2...\} \subset M^n$ such that} 
\begin{itemize}
\item [i.] \textit{for every $i\geq 0$, $\alpha(a_i,f)=\{p\}$ and $\omega(a_i,f)=\{q\}$},
\item [ii.] \textit{for every $i\geq 0$, $a_i \neq p$ and $a_i \neq q$},
\item [iii.] \textit{for every pair $i \neq j$, $i \geq 0$, $j \geq 0$,
$$
\{f^k(a_i): k \in \mathbb{Z}\} \cap \{f^k(a_j): k \in \mathbb{Z}\}= \emptyset,
$$}
\item [iv.] \textit{for every $r \geq 1$ and  $i \in \{0,1,...,r-1\}$, there exist arcs $\gamma_i$ from $a_i$ to $\sigma$, such that the sequence $\{f^k(\gamma_i)\}^{i \in \{0,...,r-1\}}_{k \in \mathbb{Z}}$ is not self-accumulated}.
\end{itemize} 
\textit{Then $h(C(f))=\infty$}.
\end{theorem}
	
\begin{proof}
Fix $r \in \mathbb{N}$. We consider $H_r \subset C(M^n)$ the set of full arcs and its induced map $C(f^{-1}):H_r \rightarrow H_r$. By Lemma ~\ref{Semiconj} we have $C(f^{-1})$ is semiconjugate to $\sigma$ defined on $(\{0,1\}^r)^{\mathbb{Z}}$. Therefore, 
$$
h(C(f)) \geq h(C(f^{-1})|_{H_r}) = r \log 2.
$$ 
Since that $r$ is fix but arbitrary, we obtain $h(C(f))= \infty$.
\end{proof}

\ack The authors want to thank Prof. D. Kwietniak for useful conversations. The second author wants to thank Prof. N. Bernardes and Phd. D. Obata.

\end{document}